\documentclass[a4paper,10pt]{article}
\usepackage{xcolor}
\usepackage{amsmath,amsthm,amsfonts,amssymb,euscript,dsfont,mathrsfs,enumerate}
\usepackage[english]{babel}
\usepackage[utf8]{inputenc}

\newcommand{\attractor}[1]{}

\DeclareUnicodeCharacter{00A0}{~}
\usepackage{hyperref}

\newtheorem{assumption}{Assumption}
\newtheorem{proposition}{Proposition}

\newtheorem{theorem}{Theorem}

\newtheorem{lemma}{Lemma}
\newtheorem{remark}{Remark}

\usepackage{geometry}
\newcommand{\R}{\mathbb{R}}
\newcommand{\drm}{\mathrm{d}}
\newcommand{\XO}{\mathcal{X}}
\newcommand{\eqdef}{:=}
\newcommand{\un}{\mathds{1}}
\newcommand{\ps}[1]{\langle #1\rangle}
\newcommand{\bP}{{\mathbb P}}
\newcommand{\bE}{{\mathbb E}}
\newcommand{\bR}{{\mathbb R}}

\newcommand{\bN}{{\mathbb N}}

\newcommand{\espf}[1]{\mathbb{E}_{\F}\left[#1\right]}
\newcommand{\Tint}{T^\mathrm{int}_{1}}
\newcommand{\clip}{{\mathrm{clip} }}

\newcommand{\cC}{{\mathcal C}}

\newcommand{\cP}{{\mathcal P}}

\newcommand{\cF}{{\mathcal F}}

\newcommand{\dr}{{d}}
\newcommand{\ex}{\mathrm{e}}

\newcommand\esp[1]{\mathbb{E}\left[#1\right]}

\newcommand\dd[2]{\frac{{d} #1}{{d} #2}}

\newcommand\p[1]{\left( #1 \right)}

\newcommand{\F}{\mathcal{F}} 

\newcommand\norm[1]{\left\lVert #1 \right\rVert}
\newcommand\abs[1]{\left\lvert #1 \right\rvert}

\DeclareMathOperator{\prox}{prox}

\usepackage{authblk}

\usepackage{fancyhdr}

\begin{document}
\pagestyle{plain}

\date{}
\title{Consensus-Based Optimization \\ Beyond Finite-Time Analysis}

\author[1]{{Pascal}~{Bianchi}}
\author[1]{{Radu-Alexandru}~{Dragomir}}
\author[1]{{Victor}~{Priser}}
\affil[1]{LTCI, T\'el\'ecom Paris}


\maketitle

\begin{abstract}
We analyze a zeroth-order particle algorithm for the global optimization of a non-convex function, 
focusing on a variant of Consensus-Based Optimization (CBO) with small but fixed noise intensity.
Unlike most previous studies restricted to finite horizons, we investigate its long-time behavior with fixed parameters. 
In the mean-field limit, a quantitative Laplace principle shows exponential convergence to a neighborhood of the minimizer \(x^\ast\). 
For finitely many particles, a analysis based on \textit{propagation of chaos} yields explicit error bounds: individual particles achieve long-time consistency near \(x^\ast\). We also analyze the performance of the ``global best'' particle. 
The proof technique combines a quantitative Laplace principle with block-wise control of Wasserstein distances, 
avoiding the exponential blow-up typical of Grönwall-based estimates.
\end{abstract}



\section{Introduction}


We aim to solve the global optimization problem  
\begin{equation*}
\min_{x\in\mathbb{R}^d} f(x)\,,
\end{equation*}
where \( f \) is non-convex, and admits a unique minimizer $x^\ast$.
Notable methods include \emph{simulated
  annealing}~\cite{pelletier1998weak}, \emph{genetic and evolutionary
  algorithms}~\cite{holland1992adaptation,fogel2006evolutionary},
\emph{Bayesian optimization}~\cite{hastings1970monte}, and swarm
intelligence methods such as \emph{Particle Swarm Optimization
  (PSO)}~\cite{Kennedy05Particle,Parsopoulos2002,Poli2007}.  PSO
methods are based on the fundamental idea of collaboration among
multiple particles, each possessing local information about the
objective function and striving to reach an optimal solution by
sharing this information with the group.  
In practice, PSO algorithms are remarkably efficient for problems
of moderate dimension.  Moreover, they rely solely on function
evaluations, which is preferable when the gradient of $f$ is difficult
to compute. However, they are considered \emph{heuristic} due to the
lack of theoretical guarantees. One promising direction to address this issue is to study a
simplified version called \emph{Consensus-Based Optimization (CBO)},
which discards both local memory and
inertia~\cite{AskariSichani2013,pinnau2017consensus,carrillo2018analytical,grassi2021particle,Fornasier_2024}. 

In its archetypal form, the CBO algorithm considers the evolution of 
\( n \) particles \((X_k^{1}, \dots, X_k^{n}) \in (\mathbb{R}^d)^n\),
defined at each iteration \( k \in \mathbb{N} \), according to the update rule
\begin{equation}
  \label{eq:cbo-intro}
  X^{i}_{k+1} = X_k^{i} + \eta_{k+1} \big( \theta_k - X^{i}_k  \big) 
                + \sqrt{2\,\eta_{k+1}}\,\sigma \xi^i_{k+1} \,,
              \end{equation}
where \(\eta_{k+1} > 0\) denotes the step size, which can be assumed fixed ($\eta_{k+1}\equiv\eta$) or decreasing, 
\(\sigma > 0\) is the noise intensity, 
and \(\xi^i_{k+1}\) are i.i.d.\ standard Gaussian random vectors.
The particles interact through the \textit{consensus point} \(\theta_k\), which is an
approximation of the best particle position at iteration \( k \).
Following the work of Pinnau \emph{et al.}~\cite{pinnau2017consensus}, it can be chosen as follows:
\begin{equation}\label{eq:def_calpha_intro}
    \theta_k \eqdef \frac{ \sum_{i=1}^n e^{-\alpha f(X^i_k)} X^i_k }{ \sum_{i=1}^n e^{-\alpha f(X^i_k)} },
\end{equation}
with \(\alpha > 0\).
For large $\alpha$, this expression provides a smooth approximation of the \emph{global best particle} (assumed here to be unique):
\begin{equation}
X^\ast_k \eqdef \arg \min \{ f(x)\,:\, x \in \{ X_k^1, \dots, X_k^n \}\} \,. \label{eq:argmink}
\end{equation}
In the literature, several variants of the CBO dynamics~(\ref{eq:cbo-intro}) have been proposed.
In \cite{Kennedy05Particle}, a random step size is considered. In \cite{pinnau2017consensus,carrillo2018analytical},
the drift term is possibly multiplied by a (smoothened) heavyside function, so that each particle $X_k^i$ is drifted towards $\theta_k$ only
if $f(\theta_k)< f(X_k^i)$. A version of CBO with adaptive step size is studied in \cite{almi2025general}.
A stabilized version of (\ref{eq:cbo-intro}) in which the consensus  point $\theta_k$ is contracted by a scaling factor is studied in \cite{huang2025uniform}.
Variants of CBO have been studied in contexts such as sampling~\cite{carrillo2021consensus}, 
constrained optimization~\cite{fornasier2020consensus,bae2022constrained}, 
multi-objective optimization~\cite{klamroth2024consensus,borghi2022consensus}, 
and the analysis of saddle points and multiple minima~\cite{fornasier2025pde,huang2024consensus}, 
with some theoretical results further extended to PSO by incorporating a personal best term~\cite{totzeck2020consensus,grassi2021particle,borghi2023consensus}.

Likewise, various definitions of the noise intensity $\sigma$ have been
proposed.  In early
works~\cite{Kennedy05Particle,Parsopoulos2002,Poli2007}, $\sigma$ is
set to zero.  In contrast, Pinnau \emph{et
  al.}~\cite{pinnau2017consensus} suggest a version similar
to~(\ref{eq:cbo-intro}), where the noise term $\sigma \xi_{k+1}^i$ is
chosen instead as ${\sigma \| X^{i}_k - \theta_k\| \,\xi_{k+1}^i}$.  By choosing a
noise intensity that decreases with the distance to \(\theta_k\), the
motivation in~\cite{pinnau2017consensus} is to ensure that the noise
vanishes as the particles approach \(\theta_k\), thereby enforcing
consensus.  The same paper also introduces a \emph{mean-field
  approximation} of the algorithm, under the form of a non-linear
Fokker-Planck equation.  This approximation is useful in the regime where
the number $n$ of particles is large, and the step size $\eta_{k+1}$
is small.  The authors of~\cite{carrillo2018analytical} consider a
mean field model similar to \cite{pinnau2017consensus}, where the
isotropic noise intensity \(\sigma \| X^{i}_k - \theta_k\|\) is replaced by
a diagonal matrix. Moreover, \cite{carrillo2018analytical} demonstrate
the well-posedness of the non-linear Fokker-Planck equation, and
proves the convergence of the mean-field particle's distribution
towards a Dirac measure, located in the vicinity of the minimizer~$x^\ast$.  This convergence result requires a certain
\emph{well-preparedness} assumption of the initial particle's
distribution. By changing the proof technique, Fornasier \emph{et
  al.} drop the well-preparedness condition in
\cite{Fornasier_2024}. 
The papers \cite{huang2022mean,gerber2023mean,Fornasier_2024}
extend the mean-field convergence results to the finite particle system, using
propagation of chaos techniques (see \cite{sznitman1984nonlinear,meleard1987propagation,Chaintron_2022}).
More recently, Fornasier \emph{et al.}~\cite{fornasier2025consensus} strengthened the theoretical 
results of~\cite{Fornasier_2024} by considering a truncated noise intensity.
Using a constant step size $\eta$, it can be shown that for a fixed $\alpha$,
if the number of iterations $k$ is bounded by a quantity of order $\mathcal{O}(\log(\alpha))$ and the number of particles $n$ is large enough,
then
\(
\bE(\|X_k^1 - x^\ast\|^2) \leq C_\alpha(\eta + \tfrac{1}{n}) + C e^{-c k},
\)
where $C, C_\alpha, c$ are some positive constants, with $C_\alpha$ depending on $\alpha$.
When $f$ has a quadratic growth around $x^\ast$, one can choose $k$ large enough, so that the righthand side becomes of order $\mathcal{O}(\frac{\log\alpha}\alpha)$ 
in the limit where $n$ is large and $\eta$ is small.
For iterations $k$ larger than $\mathcal{O}(\log(\alpha))$, this bound no longer holds, hence raising the question of the 
convergence of CBO for a large number of iterations.

In this article, we adopt the practitioner’s perspective: one first chooses the parameters 
\(n, \alpha, \eta\) of the algorithm, and then observes the convergence of the swarm as 
\(k \to \infty\). 
The natural question is: \emph{to what neighborhood of \(x^\ast\) do the particles converge?} 

Previous works do not answer this question, since in the usual analyses it is crucial 
to stop the iterations sufficiently early. 
The long-time behavior of the particles remains unknown.
This result could, in principle, be deduced from~\cite{bianchi2024longrunconvergencediscretetime}, 
but without providing a quantitative bound.

Before addressing this question, we first clarify what can reasonably be expected 
from an algorithm such as~(\ref{eq:cbo-intro}). 
Since convergence proofs rely on a mean-field analysis, any bound on 
\(\bE(\|X_k^1 - x^\ast\|^2)\) is expected to contain an irreducible term of the form 
\(C_\alpha(\eta + \tfrac{1}{n})\), as in~\cite{fornasier2025consensus}. 
Moreover, the results of \cite{ha2020convergence,ko2022convergence} suggest that
the limit of $\theta_k$ stays at a distance of order \(O(1/\sqrt{\alpha})\) from \(x^\ast\). 
This is the inherent price of replacing the global best particle with its approximation \(\theta_k\) in~(\ref{eq:def_calpha_intro}).

Based on this remark, we adopt a simpler viewpoint in this paper:
we assume that the noise intensity is fixed as \(\sigma^2 = \tfrac{\gamma}{\alpha}\), 
where \(\gamma\) is a hyperparameter (chosen large enough) and \(\alpha\) is the smoothing parameter 
appearing in the definition of \(\theta_k\).
As will be shown in this paper, we obtain a scheme in which the distance of a particle to \(x^\ast\) asymptotically fluctuates within the order of magnitude \(O(1/\sqrt\alpha)\).
Of course, the property of convergence to consensus 
is lost, however, from the optimization viewpoint, reaching an exact consensus point is not an end in itself, 
especially if the limiting consensus is not $x^\ast$.
Though minor, this simplification has substantial benefits, as it
enables the proof of long-time convergence of the algorithm as
\((k,n)\to (\infty,\infty)\). Also, for technical reasons, we consider a \emph{clipped version} of the algorithm~(\ref{eq:cbo-intro}),
in order to get a more explicit control of the particles' second order moments.
Finally, we consider vanishing step sizes of the form $\eta_{k} = \tfrac{\eta_0}{k^\zeta}$ for
some $\zeta\in (0,1]$. Then, assuming that $\alpha$ is larger than some problem-dependent constant,
we prove in Theorem~\ref{th:POC}, that (when $d\ge 6$):
\begin{equation}
\bE\left(\|X_k^i - x^\ast\|^2\right) \leq C_\alpha\left(\frac{1}{n} + \frac{1}{n^{2/d}}+\frac 1{k^\zeta}+ \frac {1}{(1+c)^{k^{1-\zeta}}} \right) + \frac C\alpha\,,\label{eq:main-short}
\end{equation}
where the constants are explicit.
This bound is valid for all \(k\). 
It shows that, as the number of iterations tends to infinity, any given particle 
approaches an \(\mathcal{O}(1/\sqrt\alpha)\)-neighborhood of \(x^\ast\), 
provided the number of particles is sufficiently large. 
We refer to this property as the \emph{\(O(1/\sqrt\alpha)\)-consistency} in \(L^2\) of the particles.

Our second main result is given in Theorem~\ref{th:best}. We prove that for every $k$ large enough,
$\bE\!\left(\|X_k^\ast - x^\ast\|^2 \right) \leq C_\alpha(n)$ where $C_\alpha(n)$ is a function of $n$
which goes to zero as $n$ goes to infinity.
We refer to this result as the \emph{consistency in $L^2$ of the global best particle}.
This result shows that \emph{even for finite \(\alpha\)}, the best particles converge to \(x^\ast\) as \((n,k)\to (\infty,\infty)\).
Comparatively, such a property cannot hold when the noise intensity 
vanishes with the distance to the consensus, as in~\cite{Fornasier_2024}, 
since in that case all particles, including the global best particle, collapse to a single point which is at distance \(O(1/\sqrt\alpha)\) from the minimizer.

\subsection*{Contributions} 
\begin{itemize}
  \item We study the long-time behavior of the CBO algorithm from the practitioner’s viewpoint, 
  where the parameters \((n,\alpha)\) are fixed and the number of iterations \(k\) tends to infinity. 
  This perspective contrasts with previous works, which require early stopping. 

  \item We introduce a simplified setting with a fixed noise variance \(\sigma^2=\gamma/\alpha\). 
  Although this removes the exact consensus property, it enables the proof of long-time convergence 
  and yields sharper consistency guarantees. 

  \item In the mean-field limit, we use a quantitative Laplace principle to relate the consensus point  
  to a proximal operator. This new viewpoint leads to an ODE approximation of the dynamics 
  whose contraction properties ensure exponential convergence to \(x^\ast\), 
  up to an \(\mathcal{O}(1/\sqrt\alpha)\) bias.

  \item For the finite-particle, discrete-time system with vanishing step sizes, 
  we develop a block-wise analysis that avoids the exponential blow-up in Grönwall’s inequality 
  and provides explicit long-time error bounds. We prove that each particle satisfies 
  \(\mathcal{O}(1/\sqrt\alpha)\)-consistency in \(L^2\) as \(k,n \to \infty\).

  \item Finally, we prove the consistency of the global best particle $X_k^*$, 
    showing that even for finite \(\alpha\), $X_k^*$ converges to a neighborhood
    of \(x^\ast\)  as $k\to\infty$, 
    whose radius can be made arbitrarily small by choosing $n$ large enough. 
  \end{itemize}


  \begin{remark}
During the final preparation of this manuscript, we became aware of the recent works 
\cite{gerber2025uniform,bayraktar2025uniform}, which establish uniform propagation of chaos for CBO over an infinite horizon. 
These results, however, do not imply consistency: without well-preparedness of the initial distribution, the long-time behavior of the mean-field system remains unknown. 
In our setting, assuming a constant noise intensity allows us to prove consistency of the mean-field dynamics (within a neighborhood of the minimizer). 
Although one might hope to combine this with their uniform propagation result to obtain particle-level consistency,
this is not possible since their analysis requires exponentially small noise in \(\alpha\), whereas we work with noise of order~\(1/\alpha\).
\end{remark}

The paper is organized as follows. In Sec.~\ref{sec:sketch}, we outline our proof technique. In Sec.~\ref{sec:main}, we state our main results, namely Theorem~\ref{th:convergence}, \ref{th:POC}, and \ref{th:best}. The proofs of these results are given in Sec.~\ref{sec:proof}. Finally, the constants appearing in our main results are made explicit in Sec.~\ref{sec:toc}.

\section{On the proof technique}
\label{sec:sketch}

\subsection{Mean field limit}
Following the pioneering papers~\cite{pinnau2017consensus,carrillo2018analytical}, 
the first step is to study the continuous-time, mean-field limit of the algorithm~(\ref{eq:cbo-intro}). 
This amounts to replacing the empirical particle distribution by a deterministic measure \(\rho_t\), 
which satisfies a nonlinear Fokker--Planck equation, also known as McKean-Vlasov equation. 
Equivalently, \(\rho_t\) is the law of a process \(X_t\) solving the SDE
\begin{equation}
  dX_t = \big( \theta_\alpha(\rho_t) - X_t \big)\,dt 
         + \sqrt{\tfrac{2\gamma}{\alpha}}\,dB_t \,,
  \label{eq:SDE}
\end{equation}
where \((B_t)_{t \ge 0}\) is a Brownian motion on \(\mathbb{R}^d\), and
\begin{equation}
  \label{eq:C}
  \theta_\alpha(\rho) 
  := \frac{\int x\, e^{-\alpha f(x)} \,d\rho(x)}{\int e^{-\alpha f(x)} \,d\rho(x)} \,,
\end{equation}
for any probability measure \(\rho\) for which the integrals are well-defined. 
The well-posedness of~\eqref{eq:SDE} is established in~\cite{carrillo2018analytical}.
However, the long-term behavior of \(\rho_t\) as \(t \to \infty\) is nontrivial.

\paragraph{The approach of \cite{Fornasier_2024}}
The strategy of~\cite{Fornasier_2024} provides an effective approximation of \(\theta_\alpha(\rho_t)\) by \(x^\ast\) 
on short time intervals, but its applicability to the long-time behavior of~(\ref{eq:SDE}) remains limited.
The key estimate, valid for any \(r>0\), is:
\[
\|\theta_\alpha(\rho) - x^\ast\| \leq C r^a 
   + \frac{e^{-\alpha L r}}{\rho(B(x^\ast,r))}\int \|x - x^\ast\|\,d\rho(x),
\]
where $B(x^\ast,r)$ is the ball of radius $r$ and center $x^\ast$.
This allows one to rewrite~(\ref{eq:SDE}) as
\[
dX_t = (x^\ast - X_t)\,dt + \varepsilon_t\,dt + \sqrt{\tfrac{2\gamma}{\alpha}}\,dB_t,
\]
with a perturbation \(\varepsilon_t\) depending on \(\rho_t(B(x^\ast,r))\) and 
\(\int \|x - x^\ast\|\,d\rho_t(x)\). 
While the latter term can be bounded, the bound grows exponentially in time, 
so the control is only effective on a finite horizon \([0,T]\). 
Beyond such short intervals, the approximation \(\theta_\alpha(\rho_t) \approx x^\ast\) breaks down. 

\paragraph{Alternative approach}
Taking expectations in~(\ref{eq:SDE}), the mean \(x_t\) of $\rho_t$ satisfies
\[
\dot x_t = \theta_\alpha(\rho_t) - x_t.
\]
Assume that the initial distribution $\nu$ is Gaussian.
With a fixed noise variance, the distribution \(\rho_t\) is Gaussian with a variance close to 
\(\sigma^2 = \tfrac{\gamma}{\alpha}\), when $t$ is large enough.
Therefore, 
\[
\theta_\alpha(\rho_t) \simeq 
  \frac{\int x\, e^{-\alpha (f(x)+\tfrac{\|x-x_t\|^2}{2\gamma})}\,dx}
       {\int e^{-\alpha (f(x)+\tfrac{\|x-x_t\|^2}{2\gamma})}\,dx}.
\]
Under suitable assumptions on \(f\), if \(x_t\) remains in a compact set and \(\gamma\) is chosen sufficiently large, the quantity
\[
\prox_{\gamma f}(x_t) \;\eqdef\; \arg\min_{x\in \bR^d} \Big(f(x)+\tfrac{\|x-x_t\|^2}{2\gamma}\Big)
\]
is well-defined and belongs to a neighborhood of $x^\ast$ on which $f$ is convex.
Then, our key argument  relies on a \emph{quantitative Laplace principle}, derived from~\cite{kirwin2010higherasymptoticslaplacesapproximation}, 
which yields
\[
\|\theta_\alpha(\rho_t) - \prox_{\gamma f}(x_t)\| \leq \tfrac{C}{\alpha},
\]
where the constant \(C\) depends only on the assumptions on \(f\) and on the diameter of the compact set containing \(x_t\). 
Thus, the ODE satisfied by the mean \(x_t\) of \(\rho_t\) takes the form
\[
\dot x_t = \prox_{\gamma f}(x_t) - x_t + \varepsilon_t,
\]
with a perturbation term \(\varepsilon_t = \mathcal{O}(1/\alpha)\). 
If \(\gamma\) is chosen large enough, the map \(x \mapsto \prox_{\gamma f}(x) - x\) is a contraction on the relevant compact set, with \(x^\ast\) as its fixed point. 
Hence, the unperturbed ODE \(\dot y_t = \prox_{\gamma f}(y_t) - y_t\) converges exponentially to \(x^\ast\). 
The same conclusion applies to \(x_t\), up to a perturbation of order \(\mathcal{O}(1/\alpha)\).

This reasoning requires control of the compact set where \(x_t\) evolves, which is possible but technically involved. 
For simplicity, we instead study a clipped version of the dynamics:
\begin{equation}
  dX_t = \big(\clip_R(\theta_\alpha(\rho_t)) - X_t \big)\,dt 
         + \sqrt{\tfrac{2\gamma}{\alpha}}\,dB_t,
  \label{eq:SDE-clip}
\end{equation}
with $X_0 \sim \nu$ and
\[
\clip_R(x) := (R \wedge \|x\|)\,\frac{x}{\|x\|},
\quad R>\|x^\ast\|.
\]
This modification provides uniform moment bounds for \(\rho_t\) at little cost, 
so that the convergence of the quadratic error follows directly:
\begin{equation}
\bE\|X_t - x^\ast\|^2 \;\leq\; \tfrac{C}{\alpha} 
   + 9 e^{-ct}\, \bE\|X_0 - x^\ast\|^2,
\label{eq:meanfield-short}
\end{equation}
for explicit constants \(C, c\). 

\subsection{Finite number of particles}
\label{sec:skFP}
We now turn to the discrete-time algorithm with \(n\) particles: 
\begin{equation}
  \label{eq:cbo-clip}
  X^{i,n}_{k+1} = X_k^{i,n} + \eta_{k+1}\big(\clip_R(\theta_k) - X^{i,n}_k\big) 
                + \sqrt{\tfrac{2\eta_{k+1}\gamma}\alpha}\, \xi^{i,n}_{k+1}\,,
\end{equation}
for \(i \in [n]\), with \(\theta_k\) defined in~(\ref{eq:def_calpha_intro}) and $X_0^{1,n},\dots,X_0^{n,n}$ are i.i.d. of law $\nu \in \mathcal{P}_2(\R^d)$. 
Note that this particle system is \textit{exchangeable}; meaning that the law of the particles is invariant by permutation and that they all have the same marginal distribution. 

A first idea would be to compare directly the law of \(X_k^{1,n}\) with the mean-field solution \(\rho_t\) 
at \(t = \sum_{j \leq k} \eta_j\), in the hope of transferring the long-time convergence of \(\rho_t\) 
to the particle system. 
However, such a comparison relies on Grönwall’s inequality, which yields bounds that grow 
exponentially with the time horizon \(T\). 
As a result, the estimate deteriorates as \(T \to \infty\), and no long-time consistency can be deduced.

To overcome this difficulty, we divide the time axis into intervals of fixed length \([jT, (j+1)T]\). 
Define
\(
k_j \;\eqdef\; \inf \Big\{ k : \sum_{i=1}^k \eta_i > jT \Big\},
\)
and denote by \(\nu^{(j)} = \sum_{i=1}^n \delta_{X^{i,n}_{k_j}} \) the empirical measure of the particles of \(X^{1,n}_{k_j}\) at iteration $k_j$. Let \(\rho_T^{(j)}\) 
the solution of the mean-field equation at time \(T\), initialized with a Gaussian distribution~\(g^{(j)}\). 
Using the triangle inequality for the Wasserstein-2 distance, we obtain
\begin{align}
  \label{eq:triangle}
  W_2^2(\nu^{(j+1)},\delta_{x^\ast}) 
   \;\leq\; 2\,W_2^2(\rho^{(j)}_T,\delta_{x^\ast}) 
           + 2\,W_2^2(\nu^{(j+1)},\rho^{(j)}_T) \,.
\end{align}
The first term is controlled by Equation~(\ref{eq:meanfield-short}):
\[
W_2^2(\rho^{(j)}_T,\delta_{x^\ast}) 
   \;\leq\; \tfrac{C}{\alpha} \;+\; 9c^{-1}  e^{-cT}\, W_2^2(\nu^{(j)},\delta_{x^\ast})\,.
\]
Choosing \(T\) such that \(9 e^{-cT} = \tfrac{1}{4}\) makes this contractive. 
The second term in~(\ref{eq:triangle}) measures the discrepancy between 
the empirical distribution at index \(k_{j+1}\) and the mean-field flow at time \(T\) initialized at \(g^{(j)}\); 
it can be bounded using Grönwall’s inequality as
\[
  \bE [ W_2^2(\nu^{(j+1)},\rho_T^{(j)}) ] \leq \Big(\bE[W_2(\nu^{(j)},g^{(j)})^2] + \tfrac{1}{n}  + \eta_{k_j}\Big) e^{C_\alpha T}
\]
We then construct a Gaussian measure $g^{(j)}$ such that $\bE[W_2(\nu^{(j)},g^{(j)})^2] \lesssim n^{-2/d}$.
 Recalling that 
\(\bE W_2^2(\nu^{(j)},\delta_{x^\ast}) = \bE(\|X_{k_j}^1-x^\ast\|^2)\), we finally arrive at
\begin{equation}
 \bE(\|X_{k_{j+1}}^1-x^\ast\|^2)
   \;\leq\; \tfrac{1}{2}\, \bE(\|X_{k_j}^1-x^\ast\|^2)
           + \tfrac{C}{\alpha}  
           +  \Big(\tfrac{1}{n} + \tfrac{1}{n^{2/d}}+ \eta_{k_j}\Big) e^{C_\alpha T},
  \label{eq:triangle-bis}
\end{equation}
for some constant $C_\alpha$ depending on $\alpha$.
Iterating this inequality over \(j\) shows that~(\ref{eq:main-short}) holds for \(k = k_j\). 
A interpolation argument extends the result to all \(k\).

\subsection{Consistency of the global best}
Our final result shows that the global best particle, defined in Equation~\eqref{eq:argmink}, 
reach an arbitrarily small neighborhood of $x^\ast$.
To this end, it suffices to bound \(\mathbb{E}\!\left[\min_{i \in [n]} \|X_k^i - x^\ast\|\right]\). 
From our previous results, for large $k$ and $n$, the law of a single particle is close to a Gaussian 
with covariance $\tfrac{\gamma}{\alpha} I_d$ and mean within $\mathcal{O}(\alpha^{-1/2})$ of $x^\ast$. 
If the particles were independent, this would easily imply that at least one of them enters an arbitrarily small neighborhood of $x^\ast$, 
thus controlling the above expectation. 
Although independence does not hold, propagation of chaos ensures that the joint law of finitely many particles 
approximates that of independent ones when $n$ is large. 
By choosing the number of particles appropriately as a function of $n$, the desired consistency follows.


\subsection*{Notation} 

Let \(\langle \cdot,\cdot \rangle\) and \(\|\cdot\|\) denote the standard Euclidean inner product 
and norm on \(\mathbb{R}^d\). 
For \(x \in \R^d\) and \(r>0\), \(B(x,r)\) denotes the open ball of radius \(r\) centered at \(x\). 
For \(a,b \in \R\), write \(a \wedge b = \min(a,b)\) and \(a \vee b = \max(a,b)\). 
For \(p \ge 1\), let \(\mathcal{P}_p(\R^d)\) be the space of probability measures 
with finite \(p\)-th moment, equipped with the Wasserstein distance \(W_p\). 
All random variables are defined on a probability space \((\Omega,\F,\bP)\). 
We write \(\mathcal{N}(\mu,\Sigma)\) for the Gaussian law on \(\R^d\) with mean \(\mu\) 
and covariance \(\Sigma\), and \(\mathcal{L}(X)\) for the law of a random variable \(X\).

\pagebreak
\section{Main results}\label{sec:main}

\subsection{Mean-field regime}
\label{sec:MF}
Let $\alpha, \gamma, R > 0$, and let \(\nu_0 \in \cP_2(\mathbb{R}^d)\). 
We study the stochastic differential equation~(\ref{eq:SDE-clip}), 
where $X_0 \sim \nu_0$, $\rho_t$ denotes the law of $X_t$, 
and the function $\theta_\alpha$ is defined in~(\ref{eq:C}). 
\begin{assumption}
  \label{hyp:f-flot}
  The function $f$ is three times continuously differentiable. Additionally, 
  \begin{enumerate}[i)]
    \item There exists $a, L> 0$ such that $\norm{\nabla f(x)} \le L(1+\norm{x}^a)$ for every $x\in\bR^d$,
    \item  \label{hypitem:growth} \( f \) admits a unique minimizer \( x^\ast \in \R^d \) and there exist constants \( \kappa > 0 \) such that \( f(x) \geq f(x^\ast) + \frac{\kappa}{2} \|x - x^\ast\|^{2} \),
    \item there exists \( \delta,\lambda>0\) such that $f$ is $\lambda$-strongly convex on \( B(x^\ast, \delta) \).
    \end{enumerate}
\end{assumption}
Under this assumption, we can guarantee the existence of a strong solution to \eqref{eq:SDE-clip}. 
Compared to earlier existence results such as \cite{carrillo2018analytical}, we impose a stronger differentiability condition on \(f\). 
This choice is mainly for convenience, as our main results will in any case rely on differentiability. 
On the other hand, the remaining assumptions are slightly weaker than in \cite{carrillo2018analytical}, due to the presence of the clipping function.

\begin{proposition}\label{prop:McKean}
  Under Assumption~\ref{hyp:f-flot}, for any initial distribution
  \(\nu_0 \in \cP_2(\mathbb{R}^d)\) and any \(T > 0\), the
  SDE~\eqref{eq:SDE-clip} admits a unique strong solution $(X_t)$ on
  \([0,T]\).
Let $y \in \bR^d$. Then for every \(t \geq 0\), there exists a Gaussian variable \(Z_t \sim \mathcal{N}(0, \tfrac{\gamma}\alpha I_d)\), independent of \(X_0\), such that
\begin{equation}\label{eq:decomp_xt}
X_t = (X_0 - y) e^{-t} + x_t + \sqrt{1-\ex^{-2t}} Z_t\,,
\end{equation} 
where \( x_t \) solves the ODE:
\begin{equation}\label{eq:xODEclip}
\dot{x}_t = \clip_R(\theta_{{\alpha}}(\rho_t)) - x_t\,, \qquad x_0 = y.
\end{equation}
Moreover, $\|x_t\| \leq \max(\|y\|,R)$ for every $t \geq 0$.
\end{proposition}
\begin{proof}
  See Sec.~\ref{sec:McKean}. 
\end{proof}

\begin{theorem}\label{th:convergence}
  Let Assumption~\ref{hyp:f-flot} hold.
  Let $\delta$ such that  $R \geq \|x^\ast\| + \delta$. 
  Assume that the initial distribution $\nu_0$ is Gaussian with mean $m_0 \in \bR^d$ and covariance matrix $\sigma^2_0 I_d$ for some $\sigma^2_0\geq \frac{\gamma}{2\alpha}$.
  Let $(X_t)_{t\ge 0}$ be the solution to~(\ref{eq:SDE-clip}).
  Then, there exists constants $\alpha_0,\bar \gamma, C_0$, such that for every $ \alpha> \alpha_0$, $\gamma>\bar \gamma$, 
  $$
\bE\norm{X_t - x^\ast}^2 \le\frac {C_0}{\alpha} +9 c_1^{-1}\,\ex^{-c_1{t}}\bE\norm{X_0 -x^\ast}^2\,,
$$
for every $t> 0$, where $c_1 \eqdef (1+\frac 2{\gamma \lambda})^{-1}$.
\end{theorem}
\begin{proof}
    The proof is provided in Sec.~\ref{sec:proofTH1}.
\end{proof}
The explicit expressions of the constants $\alpha_0,\bar \gamma, C_0$
are given in Table~\ref{table:1}.
The main strength of the bound in Theorem~\ref{th:convergence} is that it remains valid for arbitrarily large times $t$; it implies $L^2$-convergence of $X_t$ towards an 
$\mathcal{O}(1/\sqrt\alpha)$-neighborhood of $x^\ast$.


For simplicity, Theorem~\ref{th:convergence} is stated under the assumption that the initial distribution $\nu$ is Gaussian. 
Our result can be extended to more general initial laws, at the expense of introducing an additional remainder term in the proofs, 
accounting for the forgetting of the initial condition. 

\begin{remark}
One may also consider a time-dependent parameter \( \alpha_t \to \infty \). 
We conjecture that our results can be extended to this setting, 
yielding a bound on \( \mathbb{E}\|X_t - x^\ast\|^2 \) that vanishes as \(t \to \infty\). 
We leave these considerations for future work. 
\end{remark}

\subsection{Finite particle regime}
\label{sec:POC}

We now consider the discrete algorithm \eqref{eq:cbo-clip} with a finite number of particles $n$ and step size sequence $(\eta_k)_{k \ge 1}$.
\begin{assumption}
    \label{hyp:eta}
    For every $k$,    \(
    \eta_k = \frac {\eta_0}{k^\zeta}
    \),
    with $\zeta\in( 0, 1]$ and $\eta_0\in (0, 1]$.
  \end{assumption}
One could carry out the analysis under the more general assumption that 
\(\sum_{k} \eta_k = \infty\), which would in particular cover the case of a constant step size. 
Assumption~\ref{hyp:eta} is adopted here because it leads to simpler bounds. 
Moreover, since \(\eta_k \to 0\) as \(k \to \infty\), the discretization error becomes negligible in the long run, 
which enables to derive consistency results.


\begin{theorem}
\label{th:POC}
Assume that $(X_{0}^{1}, \dots, X_{0}^{n})$ are i.i.d. Gaussian r.v. with mean $m_0 \in \bR^d$ and covariance $\sigma_0^2 I_d$ where $\sigma_0^2 \ge \frac \gamma{2\alpha}$. 
Let $(X_{k}^{i})_{k\in\bN,i\in[n]}$ be the particles iterates given by Equation \eqref{eq:cbo-clip}.
Let Assumptions~\ref{hyp:f-flot} and~\ref{hyp:eta} hold. Let $R,\alpha_0, C_0$ be as in Theorem~\ref{th:convergence}.
There exists $\tilde\gamma_0>0$ and $K_0\in \bN$, such that for every $\alpha>\alpha_0$, $\gamma>\tilde\gamma_0$, $k\geq K_0$, $i\in [n]$,
\begin{equation}\label{eq:convPOC}
\bE\norm{X_{k}^{i} - x^\ast}^2\le  \frac{6C_0}{\alpha}+\frac{C_{1,\alpha}}{(1+c_{2})^{k^{1-\zeta}}} +
 \frac {C_{2,\alpha}}n + \frac{C_{3,\alpha}}{k^\zeta} + C_{4,\alpha} n^{-\theta_d} \,,\end{equation}
where the constants $C_{1,\alpha}$, $c_2$, $C_{2,\alpha}$, $C_{3,\alpha},C_{4,\alpha}$ are given in Table~\ref{table:1}, and $c_2$ does not depend on $\alpha$ and $\theta_d :=\frac 2d\vee \frac 13 $.
Moreover, it holds
\begin{equation}\label{eq:convPOC2}
\bE\norm{X_{k}^{i} - x^\ast}^2\le  \frac{6C_0}{\alpha}+ C^{-k^{1-\zeta}}+\frac{(C_{\alpha})^{k^{1-\zeta}}}n + (C_{\alpha})^{k^{1-\zeta}}\eta_0 \, ,\end{equation}
for some constants $C_\alpha$ and $C$ not made explicit here, where $C$ does not depend on $\alpha$.
\end{theorem}
\begin{proof}
    The proof is provided in Sec.~\ref{sec:ProofofTh2}.
  \end{proof}
  \begin{remark}
As in Theorem~\ref{th:convergence}, the measure $\nu$ is assumed to be Gaussian for simplicity. Moreover, Equation~\eqref{eq:convPOC2} does not possess the $n^{-\theta_d}$ bound, but this comes at the cost of diverging terms in $k$. Specifically, the bound in Equation~\eqref{eq:convPOC2} is better than that in Equation~\eqref{eq:convPOC} for a number of iterations bounded by $\mathcal{O}((\log n)^{\frac{1}{1-\zeta}})$.

  \end{remark}


Finally, consider the set of global best particles $\mathcal X_k^\ast := \arg\min\{f(x)\,:\, x\in\{X_k^{1},\dots,X_k^n\}\}$:

\begin{theorem} 
\label{th:best}
Let the conditions given in Theorem~\ref{th:POC} hold true.
For every $\alpha>\alpha_0$, $\gamma>\tilde\gamma_0$, there exists a constant $N_{\alpha}$ depending on $\alpha$ such that for every $n\ge N_{\alpha}$ there exists a constant $K_{\alpha,n}$ depending on $\alpha, n$,
such that
\begin{equation}
  \esp{ \sup_{y\in \mathcal X_k^\ast}\norm{y -x^\ast}} \le {\frac {C}{ \alpha^{\frac d{2(d+2)}} n^{c_{3,\alpha}}}}\,, \quad \forall \; k \geq K_{\alpha,n},\label{eq:bound-gb}
\end{equation}
where $C$ does not depend on $\alpha, n,k$, and $c_{3,\alpha}$ does not depend on $n,k$.
\end{theorem}
\begin{proof}
    The proof is given in Sec.~\ref{sec:propTH2}.
  \end{proof}
Theorem~\ref{th:best} shows that, as \(k \to \infty\), the global best particles converge to a neighborhood of \(x^\ast\), whose radius can be made arbitrarily small by choosing a sufficiently large number of particles. This result suggests that, after a given number of iterations, we are performing a random search over \(n^{d c_{3,\alpha}} \le n\) independent Gaussian vectors, centered in a neighborhood of radius \(\mathcal{O}(\alpha^{-1/2})\) around \(x^\ast\), with covariance matrix \(\tfrac{\gamma}{\alpha} I_d\).


\section{Proofs}
\label{sec:proof}

\subsection{Proof of Proposition~\ref{prop:McKean}}\label{sec:McKean}
 Define the functions $h_0(x) := x\exp(-\alpha f(x))$ and $h_1(x):= \exp(-\alpha f(x))$.
We first prove a key fundamental lemma.
\begin{lemma}
\label{lem:lipC}
  Let Assumption~\ref{hyp:f-flot} hold. Let ${\alpha}>0$, $\mu,\mu'\in\cP_2(\bR^d)$ and $x,y\in\bR^d$.
  We obtain:
  \begin{equation*}
  \begin{split}
          &\norm{ h_0(x) - h_0(y)} \le L_{0,\alpha}\norm{x-y} \,,\quad  | h_1(x) - h_1(y)| \le L_{1,\alpha}\norm{x-y}\,,\\
        &  \norm{\theta_{\alpha}(\mu) - \theta_{\alpha}(\mu')} \le W_1(\mu,\mu')\frac{L_{0,{\alpha}} + \norm{\theta_{\alpha}(\mu')}L_{1,{\alpha}}}{\int \ex^{-{\alpha} f(x)}\dr \mu(x)}\,. \\ 
  \end{split}
  \end{equation*}
\end{lemma}
\begin{proof}
The first two claims follow from bounding the differentials of $h_0,h_1$ using Assumption~\ref{hyp:f-flot}. For the third claim, remember the expression \eqref{eq:C} of $\theta_\alpha$. We
  remark that for $a, \Delta a \in\bR^d$ and $b,\Delta b\in \bR$ such that $b+\Delta b >0$ and $b>0$,
  \begin{equation*}
  \frac{a+\Delta a}{b + \Delta b}  - \frac ab=\frac ab\p{\frac 1{1+\frac{\Delta b}b} -1} + \frac  {\Delta a}{b+\Delta b } = \frac {\Delta a - (a/b)\Delta b}{b +\Delta b}\,.
  \end{equation*}
Let \(\pi\) be a coupling between \(\mu\) and \(\mu'\), i.e., \(\pi(\cdot \times \mathbb{R}^d) = \mu\) and \(\pi(\mathbb{R}^d \times \cdot) = \mu'\).  
We apply this result with  
\begin{equation*}
\begin{split}
   & a = \int h_0(x) \, d\pi(x,y), \quad a+\Delta a = \int h_0(y) \, d\pi(x,y)\\
  &  b = \int h_1(x) \, d\pi(x,y), \quad b+\Delta b = \int h_1(y) \, d\pi(x,y).
\end{split}
\end{equation*}
  We have
  \(
  \norm{\Delta a} \le L_{0,{\alpha}} \int \norm{x-y}\dr\pi(x,y)  \) and \(  \abs{\Delta b} \le L_{1,{\alpha}} \int \norm{x-y}\dr\pi(x,y)\,.
  \)
  Since this is true for every coupling $\pi$, the result holds.
\end{proof}
Let us prove Proposition~\ref{prop:McKean}. For establishing the existence, we use a standard fixed-point argument.
For a continuous function $t \mapsto u_t \in \R^d$, consider the SDE
\begin{equation}
      \mathrm{d}X_t = \left(u_t - X_t\right) \mathrm{d}t + \sqrt{2\frac\gamma\alpha}\, \mathrm{d}B_t
\end{equation}
initialized at $X_0 \sim \nu_0$. For any $y \in \R^d$, the solution of this SDE is
\[
  X_t^u =  (X_0 - y) e^{-t} + x_t^u + \overline{Z}_t, \qquad \overline{Z}_t = \sqrt{2\frac\gamma\alpha}\int_0^t  e^{s-t} \drm B_s,
\]
where $x_t^u$ is the solution to the ODE
\(
  \dot{x}_t^u = u_t - x_t^u
\)
initialized at $x_0^u = y$. The random variable $\overline{Z}_t$ is a centered Gaussian with covariance matrix $(1-e^{-2t})\frac\gamma\alpha I_d$.

Define $\tilde R := R\vee \norm{\esp{X_0}}$ and let $\mathcal{C}([0,T]; B(0,{\tilde R}))$ be the set of continuous functions from $[0,T]$ to $B(0,{\tilde R})$.
Define a mapping $\psi$, from $\mathcal{C}([0,T]; B(0,{\tilde R}))$ to itself, that associates to a function $u$ the function $\psi(u)$ defined by
\(
  (\psi(u))_t = \clip_{R}\left(\theta_{\alpha}(\rho_t^u)\right)
\), where $\rho_t^u$ is the law of $X^t_u$.
Note that $u=\psi(u)$ if and only if $X_t^u$ is a solution of the Mckean-Vlasov SDE \eqref{eq:SDE-clip}.

We now prove that $\psi$ admits a unique fixed point. To this end, we show that $\psi^k := \psi \circ \dots \circ \psi$ is a contraction for sufficiently large $k$.  
We rely on the following bound, which can be deduced from Lemma~\ref{lem:lipC}.
\begin{lemma}
  Let $u,u' \in \cC([0,T];B(0,\tilde R))$.
  Then, for every $t\in[0,T]$, we have
  \(
  \norm{\psi(u)_t - \psi(u')_t} \le C \norm{x_t^u-x_t^{u'}}
  \),
  where $C$ is a constant independent of $T$.
\end{lemma}
We observe that 
\(
  \dd{(x^u_t -x_t^{u'})}{t} = -(x^u_t -x_t^{u'}) + u_t -u_t'
\)
with $x^u_0 -x_0^{u'} = 0$. Solving this ODE, for every $t\in [0,T]$ yields:
\(
  x^u_t -x_t^{u'} = \int_0^t \ex^{-(t-s)}(u_s -u_s')\dr s 
\).
Then, we obtain:
\(
\sup_{s\le t}\norm{\psi(u)_s -\psi(u')_s}\le C \int_0^t \norm{u_s-u_s'}\dr s\,.
\)
Iterating the process leads to
\begin{equation*}
\sup_{s\le T}\norm{\psi^k(u)_s -\psi^k(u')_s} \le C^k \int_0^T \frac{(T-s)^k}{(k-1)!}\norm{u_s-u_s'}\dr s \le \frac{(CT)^k}{k!}\sup_{s\le T}\norm{u_s-u_s'},
\end{equation*}
which means that $\psi^k$ is a contraction for sufficiently large $k$, completing the existence proof. Finally, the bound on $x_t$ follows from a Grönwall Lemma applied to $\|x_t\|^2$.




\subsection{Proof of Theorem~\ref{th:convergence}}
\label{sec:proofTH1}
Throughout this subsection, we assume that the assumptions of Theorem~\ref{th:convergence} hold.
The proof relies on three steps. First, we establish the well posedness of a certain proximal operator. Then, we analyze the convergence of an ordinary differential equation (ODE) towards \( x^\ast \). The third step connects this ODE with the stochastic differential equation~\eqref{eq:SDE-clip}.

\subsubsection{Proximal operator}

Let \( g: \mathbb{R}^d \to \mathbb{R} \) be a convex function. For any \( m \in \mathbb{R}^d \), there exists a unique minimizer of the function  
\(
x \mapsto g(x) + \frac{1}{2} \|x - m\|^2
\).  
We denote this minimizer by \( \operatorname{prox}_g(m) \). 

\begin{lemma}
  \label{lem:prox}
  Let Assumption~\ref{hyp:f-flot} hold. Then the extended-valued function:
 \begin{equation*}
\bar f := \left\{\begin{array}{ll}
    f(x),\quad & x\in B(x^\ast,\delta), \\
     +\infty,\quad & x\notin B(x^\ast\delta) ,
\end{array}\right .
\end{equation*}
  is $\lambda$-strongly convex, proper and l.s.c.. Let $K >0$. For every $\gamma$ such that $\gamma>\frac{2K}{\delta\kappa}$ and $x\in B(x^\ast,K)$, we have
$ \arg\min f +  \frac{\|\cdot-x\|^2}{2\gamma} = \{\prox_{\gamma \bar f}(x)\}$ and 
    $\prox_{\gamma \bar f}(x) \in B(x^\ast,\delta)$.
  \end{lemma}
\begin{proof}
    

Let $K>0$ and $x\in B(x^\ast,K)$.
We search for a condition on $\gamma$ such that
\begin{equation}\label{eq:bongamma}
  f(y) + \frac{1}{2\gamma}\norm{y-x}^2 \ge f(x^\ast)+ \frac{1}{2\gamma}\norm{x^\ast-x}^2,\, \qquad \forall y\notin B(x^\ast,\delta). 
\end{equation}
Let $y\notin B(x^\ast,\delta)$, by Assumption~\ref{hyp:f-flot}, we obtain:
\begin{equation*}
  \begin{split}
    &f(y) + \frac{\norm{y-x}^2}{2\gamma} \\& \ge f(x^\ast) + \frac{\kappa}2\norm{y-x^\ast}^{2}+ \frac{1}{2\gamma}\left[\norm{y-x^\ast}^2 + \norm{x^\ast-x}^2 + 2 \langle y-x^\ast,x^\ast-x \rangle  \right]\\
    &\ge f(x^\ast)  + \frac{1}{2\gamma} \norm{y-x^\ast}\p{\kappa\gamma \delta  +\delta - 2K} + \frac{1}{2\gamma}\norm{x-x^\ast}^2\,.\\
  \end{split}
\end{equation*}
Therefore Equation~\eqref{eq:bongamma} is true under the condition that $\gamma \ge \frac{2K}{\kappa\delta}$. It follows that the minimizer of the left-hand side in $y$ necessarily belongs to $B(x^\ast,\delta)$ and 
\begin{equation*}
\underset{y\in \R^d}{\arg\min} f(y) +\frac{\norm{y-x}^2}{2\gamma} =  \underset{y\in B(x^\ast,\delta)}{\arg\min} f(y) +\frac{\norm{y-x}^2}{2\gamma}.
\end{equation*}
By the definition of $\bar f$: $\{\prox_{\gamma \bar f}(x) \} :=\underset{y\in B(x^\ast,\delta)}{\arg\min} f(y) +\frac{\norm{y-x}^2}{2\gamma}$, which ends the proof.
\end{proof}
\subsubsection{ODE satisfied by the expectation}
We now study the ODE~\eqref{eq:xODEclip} satisfied by the expectation $x_t$. We will show that it can be written as
\begin{equation}
\label{eq:odeProx}
\dot y_t = \prox_{ \gamma_t \bar f}(y_t) -y_t  +r_t \,,
\end{equation}
where $r_t$ is a (small) perturbation term, and $\gamma_t$ satisfies $\gamma_t \ge  \gamma/2$.
The contractive property of the proximal map yields the following result.
\begin{proposition}
    \label{prop:Ode}
    Let $\bar f$ be defined as in Lemma~\ref{lem:prox}.
    Let $r :\bR_+ \to \bR^d$ be a continuous function, and $\gamma:\bR_+ \to \bR $ a continuous function such that $\gamma_t\ge \gamma/2$.
    Then there exists a unique solution $t\mapsto y_t$ of Equation~\eqref{eq:odeProx} issued from $y_0\in\bR^d$. 
    Let $c_1 := \frac{1}{ (1+\frac 2{\lambda \gamma})}$.  Then, we obtain:
    \begin{equation*}
    \norm{y_t- x^\ast}^2 \le \norm{y_0 - x^\ast}^2\ex^{- t c_1  } + \frac 1{c_1}\int_0^t \norm{r_s}^2\ex^{-(t-s) c_1 }\dr s\,.
    \end{equation*}
\end{proposition}
\begin{proof}
    
Using \cite[Example~22.3 and Proposition~23.11]{bauschke2020correction}, since $ \gamma_t \bar f$ is $ \gamma_t\lambda$--strongly convex by Assumption \ref{hyp:f-flot}, the operator $\prox_{ \gamma_t \bar f}$ is $\frac 1{1+ \gamma_t\lambda}$-Lipschitz. Then the ODE~\eqref{eq:odeProx} has a unique solution and defining $V(x) :=\frac{1}{2} \norm{x-x^\ast}^2$, we obtain:
\begin{equation*}
  \begin{split}
    \frac{d V(y_t)}{dt} &= \ps{y_t-x^\ast,\prox_{ \gamma_t \bar f}(y_t)-y_t} + \ps{y_t -x^\ast,r_t} \\
    &= - \|y_t-x^\ast\|^2 + \ps{y_t-x^\ast, \prox_{ \gamma_t \bar f}(y_t)-x^\ast} + \ps{y_t -x^\ast,r_t}\\
    &\leq - \frac{\lambda  \gamma_t}{1 + \lambda  \gamma_t } \|y_t-x^\ast\|^2 +\norm{r_t}\|y_t-x^\ast\| \\
   & \le -V(y_t)(\frac{2}{1+\frac 1{\lambda \gamma_t}} -\frac{1}{1+\frac 1{\lambda \gamma_t}} ) + \frac { 1+\gamma_t\lambda}{2 \gamma_t\lambda}\norm{r_t}^2.
  \end{split}
\end{equation*}
We use Grönwall's lemma together with the bound $\gamma_t \ge \gamma/2$ to complete the proof.
\end{proof}

\subsubsection{The rest term}
Recall that the solution of \eqref{eq:SDE-clip} writes $X_t = (X_0-x^\ast)\ex^{-t} +x_t + \sqrt{1-\ex^{-2t}}Z_t$, where $Z_t \sim \mathcal{N}(0,\frac{\gamma}{\alpha}I_d)$ is independent of $X_0\sim \mathcal N(m_0,\sigma_{0}^2I_d)$, and $x_t$ satisfies the ODE~\eqref{eq:xODEclip} with $x_0 = x^\ast$. Therefore, its law $\rho_t$ is the Gaussian distribution $\mathcal{N}(m_t, \frac{\gamma_t}{\alpha} I_d)$ with $m_t := (m_0 - x^*)e^{-t} + x_t$ and $ \gamma_t :=  \alpha \sigma_0^2 \ex^{-2t} +(1-\ex^{-2t}) \gamma$.

Equivalently, $x_t$ solves the perturbed ODE~\eqref{eq:odeProx} with
\begin{equation}\label{eq:pert_term}
\begin{split}
  r_t &=\clip_R(\theta_{\alpha}(\rho_t)) - \prox_{\gamma_t \bar f}(x_t)\\  
        &= \clip_R(\theta_{\alpha}(\rho_t)) - \clip_R(\prox_{\gamma_t \bar f}(m_t ))+ \\ \clip_R&(\prox_{  \gamma_t \bar f}(m_t )) - \clip_R(\prox_{\gamma_t \bar f}(x_t )) + \clip_R(\prox_{\gamma_t \bar f}(x_t )) - \prox_{\gamma_t \bar f}(x_t ).
\end{split}
\end{equation}

We bound the first term using the \emph{Laplace principle}. 

\begin{lemma}[Laplace principle]
  \label{lem:laplace}

Let $t,K > 0$. Assume that $\gamma>\frac{4K}{\kappa\delta}$ and $m_t\in B(x^\ast,K)$.
Then, there exists a constant $C^{\mathrm{Lap}}$ depending only on $f,\gamma$, such that
\(
\left\|\theta_\alpha(\rho_t)- \prox_{\gamma_t\bar f}(m_t) \right\|\leq \frac{C^{\mathrm{Lap}}}{{\alpha}}\,.
\)
\end{lemma}
\begin{proof}
Recalling the definition \eqref{eq:C}, note that we have \(
\theta_\alpha(\rho_t) = \frac{\int y \, e^{-\alpha \phi_t(y) } \, \mathrm{d}y}{\int e^{-\alpha \phi_t(y)} \, \mathrm{d}y}\,
\) with $\phi_t(y) := f(y) + \frac{\|y - m_t\|^2}{2\gamma_t}$. Since we assumed $\sigma_0^2 \geq \gamma / {2\alpha}$, we have $\gamma_t \geq \gamma /2 \geq 2K /(\kappa\delta)$; it follows from Lemma~\ref{lem:prox} that $\phi_t$ admits a unique minimizer: $\operatorname{prox}_{\gamma_t \bar{f}}(m_t)$. Then we use the quantitative Laplace principle of~\cite[Corollary 3.4]{kirwin2010higherasymptoticslaplacesapproximation}, which states that the distance between $\theta_\alpha(\rho_t)$ and the minimizer of $\phi_t$ is bounded by $C^{\mathrm{Lap}}/\alpha$.
\end{proof}
\begin{lemma}\label{lem:lapLipschitz}
  \(
  \norm{\prox_{\gamma_t\bar f}(x_t) - \prox_{\gamma_t\bar f}(m_t)} \le  \norm{m_0 -x^\ast}\ex^{-t}
  \).
\end{lemma}
\begin{proof}
    The proximal operator of a convex function is $1$-Lipschitz \cite{bauschke2020correction}.
\end{proof}

\begin{lemma}\label{lem:proxInB}
Assume that $R\ge \norm{x^\ast} +\delta$.
Let $K>0$ and let $\gamma >\frac{4K}{\delta \kappa}$.
For every $x\in B(x^\ast,K)$,
\(\clip_R(\prox_{\gamma_t \bar f}(x )) - \prox_{\gamma_t \bar f}(x ) = 0\,.\)
\end{lemma}
\begin{proof}
    By Lemma~\ref{lem:prox}, we have $\norm{\prox_{\gamma \bar f}(x )}\le \norm{x^\ast}+\delta\le R$.
\end{proof}

Note now that, thanks to the clipping function we have $\|x_t\| \leq R$ (Proposition \ref{prop:McKean}) and therefore we get $\| x_t \| \vee \|m_t\| \leq K$ for every $t$, with $K:= \|m_0-x^*\| \vee R$. Define $\bar \gamma = 4K/(\delta \kappa)$ and assume from now on that $\gamma \geq \bar \gamma$.
We can then apply the three preceding lemmas to bound the perturbation term in \eqref{eq:pert_term}, using the fact that $\clip_R$ is 1-Lipschitz (since it is the projection on the $\ell_2$ ball \cite{bauschke2020correction}):
\[
  \|r_t\| \leq  \frac{C^{\mathrm{Lap}}}{\alpha} +  \ex^{-t}\norm{m_0 -x^\ast}.
\]
\subsubsection{Finalizing the proof of Theorem~\ref{th:convergence}}
By combining Proposition~\ref{prop:Ode} (with $y_0 = x^*$) with the bound on $r_t$, we get, remembering that $c_1< 1$: 
\begin{equation*}
\begin{split}
    \norm{x_t -x^\ast}^2 &\le \frac 1{c_1}\int_0^t e^{-(t-s)c_1} \| r_s \|^2 ds\le
   \frac {2(C^\mathrm{Lap})^2}{\alpha^2 (c_1)^2} +  \frac{2 \ex^{-c_1t}}{c_1} \norm{ m_0-x^\ast}^2.
\end{split}
\end{equation*}
Choosing $\alpha> \alpha_0 := \frac{2(C^\mathrm{Lap})^2}{c_1^2\gamma d}$, using Proposition \ref{prop:McKean} and with $C_0 :=6\gamma d$, it holds: 
\begin{equation*}
\begin{split}
\bE \|X_t -x^\ast\|^2 &\le 3\|x_t-x^\ast\|^2 + \frac{3\gamma d}{\alpha} + 3e^{-2t} \bE \|X_0-x^\ast\|^2 \\
&\le \frac{6\gamma d}{\alpha} + 6 c_1^{-1} e^{-c_1 t} \|m_0-x^\ast\|^2 + 3\ex^{-2t} \bE\norm{X_0-x^\ast}^2 \\
&\le \frac{C_0}{\alpha} + 9c_1^{-1}\ex^{-c_1t} \bE\norm{X_0-x^\ast}^2\,.
\end{split}
\end{equation*}

\subsection{Proof of Theorem~\ref{th:POC}}
\label{sec:ProofofTh2}

Let \( (B^{i,n}_t)_{t \ge 0,\, i \in [n]} \) be \( n \) independent standard Brownian motions in $\bR^d$, and $(\XO^{i,n})_{i \in [n]}$ be $n$ r. v. on $\bR^d$ which are independent from those Brownians.
We define a system of $n$ stochastic differential equations, describing the behavior of CBO in the \textit{finite-particle} and \textit{continuous-time regime}:
\begin{equation}\label{eq:SDEfinite}
    \dr X_t^{i,n} = (\clip_R(\theta_\alpha(\mu^{n}_t))-  X_t^{i,n})\dr t + \sqrt{2\frac\gamma\alpha}\dr B^{i,n}_t \text{ and } X_0^{i,n} = \XO^{i,n} \,,
\end{equation}
for $i \in [n]$,
where $\mu_t^{n} = \frac 1n \sum_{i=1}^n \delta_{ X_t^{i,n}}$ is the occupation measure.

Similarly, consider \( n \) independent standard Brownian motions \( (\bar B^{i,n}_t)_{t \ge 0,\, i \in [n]} \) and $n$ random variables $(\bar \XO^{i,n})_{i \in [n]}$ which are independent from those Brownians.
We define $n$ copies of the \textit{mean-field} SDE \eqref{eq:SDE-clip} as:
\begin{equation}\label{eq:SDEfinite_bar}
    \dr\bar X_t^{i,n} = (\clip_R(\theta_\alpha(\rho^{i,n}_t))- \bar X_t^{i,n})\dr t + \sqrt{2\frac\gamma\alpha}\dr \bar B^{i,n}_t \text{ and }
    \bar X_0^{i,n} = \bar \XO^{i,n},
\end{equation} 
for $i \in [n]$, where, $\rho_t^{i,n} =\mathcal L(\bar X_t^{i,n})$. Importantly, the solutions $(\bar X_t^{i,n})_{i\in[n]}$ of \eqref{eq:SDEfinite_bar} are i.i.d.\ whenever the initial variables $(\bar \XO^{i,n})$ are i.i.d.

First, we give bounds on the second moments obtained thanks to the clipping function. We define:  $(C_5)^2 :=(2 m_0^2 + 2R^2 + (3d+1)\sigma_0^2)\vee \p{R^2 + 2d\frac\gamma\alpha} )$, where we recall that $m_0$ is the mean and $\sigma_0^2 I_d$ the covariance of the initial iterates of Algorithm~\eqref{eq:cbo-clip}. 
\begin{lemma}
\label{lem:stab}
Assume that we have $\esp{\|\XO^{i,n}\|^2]} \leq (C_5)^2$ and $\esp{ \| \bar\XO^{i,n}\|^2]} \leq (C_5)^2$ for every $i \in [n]$.
Then the solutions of \eqref{eq:SDEfinite} and \eqref{eq:SDEfinite_bar} satisfy for $i \in [n]$:
    $$  
    \sup_{t\in\bR_+} \p{\esp{\norm{ X^{i,n}_t}^2}\vee \esp{\norm{\bar X^{i,n}_t}^2}}  \le (C_5)^2\, \mathrm{and}\,    \inf_{t\in\bR_+}\esp{h_1(\bar X^{i,n}_t)} \ge C^\mathrm{int}_{8,\alpha}\,,
    $$
    with the constant $C^\mathrm{int}_{8,\alpha} :=\frac{3}{4}\exp(-\alpha\sup_{y\in B(0,2C_5) } f(y) ) $.
\end{lemma}

\begin{proof}
Using Itô's' formula, we obtain that
\(
\dd{\bE\norm{X^{i,n}_t}^2}{t} \le R^2 -\bE\norm{X_t^{i,n}}^2 + 2d\frac\gamma\alpha
\).
Hence by Grönwall's Lemma, we get (and similarly for $\bar X^{i,n}_t$):
\begin{equation*}
\bE\norm{X_t^{i,n}}^2 \le \p{\bE\norm{X_0^{i,n}}^2 -R^2 -2d\frac\gamma\alpha }\ex^{-t}  + {R^2 +2d\frac\gamma\alpha}.
\end{equation*}
Then we use $\alpha >\alpha_0$. The second claim is deduced from Markov's inequality.
\end{proof}



The proof requires two intermediate steps: first, a propagation of chaos result in Sec.~\ref{sec:Proofpoc}, and second, a bound on the Euler approximation in Sec.~\ref{sec:ProofEuler}.
\subsubsection{Propagation of chaos}
\label{sec:Proofpoc}
We show that, when $n$ is large enough, the solutions to \eqref{eq:SDEfinite} and \eqref{eq:SDEfinite_bar} are close. We do so by coupling the equations through an adequate choice of the Brownians $(\bar B_t^{i,n})$. We define $\theta_d := \frac 2d\vee \frac 13$.

\begin{proposition}\label{prop:POC}
  Consider the solutions $(X^{i,n}_t)_{i \in [n]}$ of the system \eqref{eq:SDEfinite}. Assume that the initial variables $(\XO^{i,n})_{i \in [n]}$ are measurable with respect to a $\sigma$-algebra $\F$, that the event $B:= \{ n^{-1}\sum_{i=1}^n \|\XO^{i,n}\|^2 \leq (C_5)^2 \}$ satisfies $\bP(B^c) \le \frac {2d}n$, and that
 $ \bE[\|\XO^{i,n}\|^4] \leq (C^\mathrm{int}_{17})^2$ for every $i\in[n]$.
  
  Let $\nu_0:\Omega \to \mathcal{P}_2(\bR^d)$ be a random measure which is $\F$-measurable. Assume that 
  $\un_B \int \norm{x}^2 \dr\nu_0(x) \leq (C_5)^2$ almost surely and $\bE\int\norm{x}^4\dr\nu_0(x) \le (C_{17}^\mathrm{int})^2$.
  Then, there exist, conditionally on $\F$, $n$ i.i.d solutions $(\bar X_t^{i,n})_{i\in[n]}$ of the McKean--Vlasov equation~\eqref{eq:SDEfinite_bar} with initial law $\nu_0$ such that 
  \[
   \bE\left[W_2(\mu_t^n,\bar \mu_t^n)^2\right] \le \, \frac{C^\mathrm{int}_{2,\alpha}}{n}\, e^{C^\mathrm{int}_{3,\alpha} t} \! +\!  \left( 2\mathbb{E}\!\left[ W_2\,(\mu_0^{n},\nu_0)^2 \right] + \frac{C^\mathrm{int}_{16}} {n^{\theta_d}}\right) e^{C^\mathrm{int}_{3,\alpha} t},
  \]
  for all $t\geq 0$, where $\mu_t^n = \frac{1}{n}\sum_{i=1}^n \delta_{X_t^{i,n}}$ 
  and $\bar \mu_t^n = \frac{1}{n}\sum_{i=1}^n \delta_{\bar X_t^{i,n}}$. 
  
  When, additionally, $(\XO^{i,n})$ are i.i.d. with law $\tilde{\nu_0}$ and $\nu_0$ is deterministic, we have 
  \begin{equation*}
  \begin{split}
    \bE\left[W_2(\mu_t^n,\bar \mu_t^n)^2\right]
    \le \frac{C^\mathrm{int}_{2,\alpha}}{n}\, e^{C^\mathrm{int}_{3,\alpha} t} 
     +  W_2(\tilde \nu_0,\nu_0)^2 e^{C^\mathrm{int}_{3,\alpha} t} \,.
  \end{split}
  \end{equation*}
\end{proposition}

\begin{proof}
Let us work conditionally on the $\sigma$--algebra $\F$. Define $(\bar \XO^{i,n})_{i \in [n]}$ to be $n$ i.i.d random variables with law $\nu_0$.
Writing $W_2\!\left( \mu_0^n, \bar \mu_0^n\right) \leq W_2\!\left( \mu_0^n, \nu_0 \right)+W_2\!\left( \nu_0, \bar \mu_0^n\right)$ and using Lemma~\ref{lem:LLN}, we obtain
\begin{equation}
\label{eq:ini}
\begin{split}
  \bE_{\F}[W_2\!\left( \mu_0^n, \bar \mu_0^n\right)^2]
  &\leq 2  W_2(\mu_0^{n},\nu_0)^2  + 2 C^\mathrm{LLN}  \p{\int \norm{x}^4 \dr \nu_0(x)}^{\frac 12} n^{-\theta_d}.
\end{split}
\end{equation}

By \cite[Proposition~2.1]{peyre2019computational}, for every realization of the variables $(\bar \XO^{ i,n})_{i \in [n]}$ there exists a permutation $\tau$ such that
\[
\frac{1}{n}\sum_{i=1}^n \| \XO^{i,n}-\bar \XO^{\tau(i),n} \|^2 
  = W_2\!\left( \mu_0^n, \bar \mu_0^n\right)^2.
\]
This defines thus a random permutation~$\tau$.
Consider now the Brownian motions ${\bar B_t^{i,n} = B_t^{\tau^{-1}(i),n}}$ for $i \in [n]$. Since $\tau$ is independent of $(B^{i,n}_t)$, the permuted family $(\bar B^{i,n}_t)$ is still a family of $n$ independent Brownian motions. Besides, since $(B_t^{i,n})$ are independent from $( \XO^{i,n})$ and $(\bar \XO^{i,n})$, so is the family $(\bar B_t^{i,n})$.

Consider (still conditionally on $\F$) the solutions $(\bar X_t^{i,n})_{i\in [n]}$ of the equations~\eqref{eq:SDEfinite_bar} initialized at $(\bar \XO^{i,n})_{i \in [n]}$ and driven by $(\bar B_t^{i,n})$, which are i.i.d.

Define the quantities
\(
\Delta_t^{i} := \bar X_t^{\tau(i),n} - X_t^{i,n}
\).
Observe that $\Delta_t^i$ satisfies the SDE \[ \drm \Delta_t^i = \left[\clip_R(\theta_{\alpha}(  \rho_t)) - \clip_R(\theta_{\alpha}( \mu^n_t)) - \Delta_t^i\right] \drm t.
\]
Define the event $B = \{ n^{-1}\sum_{i=1}^n \|\XO^{i,n}\|^2 \le (C_5)^2\} \in \F$. Then we have
\begin{equation*}
\begin{split}
 & \dd{}{t} \bE_{\F}\p{\frac 1n \sum_{i=1}^n\norm{\Delta_t^i}^2} 
   \le 2\un_B \bE_{\F}\p{{\frac 1n \sum_{i=1}^n \norm{\clip_R(\theta_{\alpha}(  \rho_t)) - \clip_R(\theta_{\alpha}(\bar \mu^n_t))}^2 }}+\\
  &2\un_B\bE_{\F}\p{{\frac 1n \sum_{i=1}^n \norm{\clip_R(\theta_{\alpha}( \bar \mu^n_t)) - \clip_R(\theta_{\alpha}(\mu^n_t))}^2 }} + 4R^2\un_{B^c}-  {\bE_{\F}\p{\frac 1n \sum_{i=1}^n\norm{\Delta_t^i}^2}}\\
  & := 2\un_B \chi_1 +2\un_B \chi_2 + 4R^2\un_{B^c} -  {\bE_{\F}\p{\frac 1n \sum_{i=1}^n\norm{\Delta_t^i}^2}} 
  \,.
\end{split}
\end{equation*}
We proceed to bound the two terms $\chi_1,\chi_2$, starting with $\chi_2$. 
By Lemma~\ref{lem:lipC}, 
\begin{equation*}
\norm{ \theta_{\alpha}(  \mu^n_t) - \theta_{\alpha}(\bar \mu^n_t)} \le \frac{\frac{1}n \sum_{i=1}^n \norm{\Delta_t^i}}{{\frac 1n \sum_{i=1}^n h_1( X_t^{i,n} )}}  \p{L_{0,{\alpha}} + \norm{\theta_{\alpha}(\bar \mu^n_t)}L_{1,{\alpha}}  }\,.
\end{equation*}
To simplify, write $C:= 2C_5$ and $c:= \frac{1}2 C^\mathrm{int}_{8,\alpha}$.
We define the random event $A_1 := \{ \frac 1n\sum_{i=1}^n h_1(X_t^{i,n}) > c, \frac 1n\sum_{i=1}^n h_1(\bar X_t^{i,n}) > c, \frac 1n\sum_{i=1}^n \norm{h_0(\bar X_t^{i,n})} \le C \}$. Then, we have
\(
{\norm{ \theta_{\alpha}(  \mu^n_t) - \theta_{\alpha}(\bar \mu^n_t)}\un_{A_1} } \le   \frac{1}n \sum_{i=1}^n \norm{\Delta_t^i}  \p{\frac{L_{0,{\alpha}} + \frac Cc L_{1,{\alpha}}  }{c}}.
\)
Therefore:
\begin{equation*}
    \chi_2   \le  \p{\frac{L_{0,{\alpha}} + \frac Cc L_{1,{\alpha}}  }{c}}^2 \bE_{\F} \p{\frac 1n \sum_{i=1}^n\norm{\Delta_t^i}^2} +  4R^2\bP_{\F}(A_1^c)\,.
\end{equation*}
Remark that 
\begin{multline*}
  \bP_{\F}(A_1^c) \le \bP_{\F}\p{\frac 1n\sum_{i=1}^n h_1(X_t^{i,n}) \le c} + \bP_{\F}\p{\frac 1n\sum_{i=1}^n \norm{h_0(\bar X_t^{i,n})} > C} \\+\bP_{\F}\p{\frac 1n\sum_{i=1}^n h_1(\bar X_t^{i,n}) \le c}.
\end{multline*}
We have $\un_B \espf{\|\bar \XO^{i,n}\|^2} \leq (C_5)^2$ hence $\un_B\espf{h_1(\bar X_t^{1,n})} > 2c$ by Lemma~\ref{lem:stab}, and
\begin{equation}
\label{eq:h1}
\begin{split}
  \un_B\bP_{\F}\p{\frac 1n\sum_{i=1}^n h_1(\bar X_t^{i,n}) <c} &\le \bP_{\F}\p{\frac 1n\sum_{i=1}^n (h_1(\bar X_t^{i,n}) - \espf{h_1(\bar X_t^{1,n})}) <c - 2c} \\
  &\le  \bP_{\F}\p{\p{\frac 1n\sum_{i=1}^n (h_1(\bar X_t^{i,n}) - \espf{h_1(\bar X_t^{1,n})})}^2 \ge c^2}\\
  & \le \frac{\espf{\p{ {h_1(\bar X_t^{1,n}) - \espf{h_1(\bar X_t^{1,n})}}}^2}}{nc^2}\\
  & \le \frac{2}{nc^2},\\
\end{split}
\end{equation}
remembering that $(\bar X_t^{1,n})_{i\in[n]}$ are i.i.d.
In the same way, since $\un_B\bE_{\F}\norm{h_0(\bar X_t^{1,n})} \le \frac C2$ by Lemma~\ref{lem:stab}, we obtain:
\(
   \bP_{\F}\p{\frac 1n\sum_{i=1}^n \norm{h_0(\bar X_t^{i,n})} > C} 
   \le \frac 2n
\).
Finally,
\begin{equation*}
 \un_B \bP_{\F}\p{\frac 1n\sum_{i=1}^n h_1(X_t^{i,n}) <c} 
  \le \frac{\un_B \espf{\p{ \frac 1n\sum_{i=1}^n \p{h_1( X_t^{i,n}) - \espf{h_1(\bar X_t^{1,n})}}}^2}}{c^2}\,.\\
\end{equation*}
Using Equation~\eqref{eq:h1} and the triangle inequality  together with $\sum_{i\in[n]} h_1(\bar X^{\tau(i),n}_t) =\sum_{i\in[n]} h_1(\bar X^{i,n}_t) $, we obtain that:
\begin{equation*}
\begin{split}
  &\un_B\espf{\p{ \frac 1n\sum_{i=1}^n \p{h_1( X_t^{i,n}) - \espf{h_1(\bar X_t^{1,n})}}}^2}
  \le   2{L_{1,{\alpha}}^2} \espf{\frac 1n \sum_{i=1}^n\norm{\Delta_t^i}^2} + \frac 4{n}\,.
\end{split}
\end{equation*}
Then, we get:
\(
    \un_B\bP(A^c_1) \le  \frac{2}n \p{1+\frac{3}{c^2}} + \frac{2L_{1,\alpha}^2}{c^2} \espf{\frac 1n \sum_{i=1}^n\norm{\Delta_t^i}^2}
\) and we obtain:
\begin{equation*}
  \begin{split}
  \un_B \chi_2& \le \p{\p{\frac{L_{0,{\alpha}} + \frac Cc L_{1,{\alpha}}  }{c}}^2+ \frac{8R^2L_{1,{\alpha}}^2}{c^2}}  \espf{\frac 1n \sum_{i=1}^n\norm{\Delta_t^i}^2} + \frac{8R^2}{n}{ \p{1+\frac{3}{c^2}} }\,.
  \end{split}
\end{equation*}
We now bound $\chi_1$, noting that
\begin{align*}
  &\theta_\alpha(\bar\mu^n_t) -\theta_\alpha(\rho_t)
  = \frac{
    \frac 1n \sum_{i=1}^n h_0(\bar X_t^{i,n})}{ \frac 1n \sum_{i=1}^n h_1(\bar X_t^{i,n}) } - \frac{  \espf{h_0(\bar X^{i,n}_t)}}{ \espf{h_1(\bar X^{i,n}_t)}} \\
  &= \frac{
  \frac 1n \sum_{i=1}^n (h_0(\bar X^{i,n}_t)-\espf{h_0(\bar X^{i,n}_t)})  - \theta_\alpha(\rho_t) \frac 1n \sum_{i=1}^n (h_1(\bar X^{i,n}_t)-\espf{h_1(\bar X^{i,n}_t)})
  } {
  \frac 1n \sum_{i=1}^n h_1(\bar X_t^{i,n})
  }\,.
\end{align*}
We define the event $A_2 := \{ \frac 1n\sum_{i=1}^n h_1(\bar X^{i,n}_t) >c \}$, and we obtain that:
\begin{multline*}
    \espf{\un_{A_2}\norm{\theta_\alpha(\bar\mu^n_t) -\theta_\alpha(\rho_t)}^2} \le  \frac{2}{nc^2}\times\\ \p{\bE_{\F}\norm{h_0(\bar X^{1,n}_t)-\espf{h_0(\bar X^{1,n}_t)}}^2 +  \norm{\theta_\alpha(\rho_t)}^2\bE_{\F}\big(h_1(\bar X^{1,n}_t)-\bE_{\F}{h_1(\bar X^{1,n}_t)}\big)^2}.
\end{multline*}
By Lemma~\ref{lem:stab}, $\un_B\norm{\theta_{\alpha}(\rho_t)} \le \frac{C_5}{C^\mathrm{int}_{8,\alpha}}$, which yields:
\(
    \un_B\espf{\un_{A_2}\norm{\theta_\alpha(\bar\rho^n_t) -\theta_\alpha(\rho_t)}^2} \le \frac{4}{nc^2}\p{(C_5)^2 +\p{\frac{C_5}{C^\mathrm{int}_{8,\alpha}}}^2 }. 
\)
Coming back to $\chi_1$, we obtain:
\begin{equation*}
\begin{split}
   \un_B\chi_1 
   &\le  \un_B  \espf{\un_{A_2}{\frac 1n \sum_{i=1}^n \norm{\clip_R(\theta_{\alpha}(  \rho_t)) - \clip_R(\theta_{\alpha}(\bar \rho^n_t))}^2 }}  + 4R^2 \bP_{\F}(A_2^c)\un_B\\
  & \le 4R^2 \un_B\bP_{\F}(A_2^c)\un_B +  \frac{4}{nc^2}\p{(C_5)^2 +\p{\frac{C_5}{C^\mathrm{int}_{8,\alpha}}}^2 }\\
  &\le \frac {16}{n\p{C^\mathrm{int}_{8,\alpha}}^2}\p{2R^2 +(C_5)^2 +\p{\frac{C_5}{C^\mathrm{int}_{8,\alpha}}}^2  } \,,
\end{split}
\end{equation*}
where, using Equation~\eqref{eq:h1}, we obtained $\un_B\bP_{\F}(A_2^c) \le \frac 2{nc^2}$ and we also used $c = \frac 12C^\mathrm{int}_{8,\alpha}$.
Now combining all inequalities, we obtain:
\begin{equation*}
  \dd{\bE_{\F}\p{\frac 1n\sum_{i=1}^n\norm{\Delta_t^i}^2}}{t} \le  C^\mathrm{int}_{3,\alpha} \bE_{\F}\p{\frac 1n\sum_{i=1}^n\norm{\Delta_t^i}^2} + \frac{C^\mathrm{int}_{2,\alpha} -8R^2d}{n}   + 4R^2\un_{B^c}\,,
\end{equation*}
with the constants $C^\mathrm{int}_{2,\alpha}$ and $C^\mathrm{int}_{3,\alpha}$ defined in Table~\ref{table:3}.
We can then apply Grönwall's lemma, bound the initial quantity with Equation~\eqref{eq:ini}, and conclude by taking the expectation of the conditional expectation with respect to~$\cF$, with $\bP(B^c) \le \frac {2d}n$. We finally note that $W_2(\mu_t^n,\bar \mu_t^n)^2 \le n^{-1}\sum_{i=1}^n \|\Delta_t^i\|^2$.

For the second claim of the proof, define the initial variables $(\bar \XO^{i,n})_{i \in [n]}$ to be such that for every $i$, $(X^{i,n}_0,\bar X^{i,n}_0)$ forms the optimal coupling between $(\tilde{\nu}_0,\nu_0)$, and consider the Mckean--Vlasov equations \eqref{eq:SDEfinite_bar} with $\bar B_t^{i,n} = B_t^{i,n}$. We then have, at initialization, $\esp{n^{-1}\sum_{i=1}^n \|\Delta_0^i\|^2} = W_2(\tilde{\nu}_0,\nu_0)$, and the rest of the proof is the same.
\end{proof}

\subsubsection{Bounding the discretization error}
\label{sec:ProofEuler}
We now turn our attention to the discrete processes $(X_{k,\eta}^{i,n})_{k \ge 0,\, i \in [n]}$ which satisfy, following Equation \eqref{eq:cbo-clip}:
\[
X_{k+1,\eta}^{i,n} = X_{k,\eta}^{i,n}  
 + \eta_{k+1}\big( \text{clip}_R(\theta_\alpha(\mu_{k,\eta}^n)) -X_{k,\eta}^{i,n}\big)
 + \sqrt{2\frac{\gamma \eta_{k+1}}{\alpha}}\, \xi_{k+1}^{i,n},
\]
where $\mu_{k,\eta}^{n} := \tfrac{1}{n} \sum_{i \in [n]} \delta_{X_{k,\eta}^{i,n}}$, the noise variables $(\xi_{k}^{i,n})_{k \ge 0,\, i \in [n]}$ are i.i.d. standard Gaussian and the processes are initialized i.i.d. with distribution $\nu = \mathcal{N}(m_0,\sigma_0^2 I_d)$.

The next lemma states the stability of the particles \( (X_{k,\eta}^{i,n}) \); it is the discrete counterpart of Lemma~\ref{lem:stab}. 
Among other things, it allows us to verify the assumptions of Proposition~\ref{prop:POC}.
\begin{lemma}
\label{lem:stabDIS}
For every $k \geq 0$, there exists a random variable $x^n_k$ and $n$ i.i.d. variables  $Z^{1,n}_k, \dots, Z^{n,n}_k \sim\mathcal N(0, I_d)$, such that
    $
   X^{i,n}_{k,\eta} = x^n_k + \sigma_k Z_k^{i,n}\,,
    $
with $\sigma_0^2 \vee \frac{2\gamma}{\alpha} \ge \sigma_k^2 \ge\frac \gamma\alpha$, and $\norm{x^n_k}^2 \le \frac 12(C_5)^2 - (d+1)\sigma_0^2 $ almost surely.

Moreover, we have $\bE \norm{X^{i,n}_{k,\eta}}^2 \le (C_5)^2$ and $\bE \norm{X^{i,n}_{k,\eta}}^4 \le (C^\mathrm{int}_{17})^2$.
Finally, for every $k\in\bN$, the event $B = \{\frac 1n\sum_{i\in[n]}\norm{X_{k,\eta}^{i,n}}^2 \le (C_5)^2 \}$ satisfies
$\bP(B^c) \le \frac {2d}n $.
\end{lemma}
\begin{proof}
 Define the sequence $(x^n_k) $ with $x_0 = m_0$ and 
 \[ x^n_{k+1} =(1-\eta_{k+1}) x^n_k + \eta_{k+1} \clip_R(\theta(\mu_{k,\eta}^n)), \quad \forall k \geq 0.
 \]
  By induction, as $\eta_k \in (0,1)$, we have $\norm{x^n_k}^2 \le R^2 \vee \norm{m_0}^2 \leq \frac{1}{2}(C_5)^2 - (d+1)\sigma_0^2$. Defining $\tilde{Z}_k^{i,n} = X^{i,n}_{k,\eta} - x^n_k$, we have $\tilde{Z}_k^{i,n} \sim \mathcal{N}(0,\sigma_0^2 I_d)$ and \[\tilde{Z}_{k+1}^{i,n} = (1-\eta_{k+1})\tilde{Z}_k^{i,n} + \sqrt{\frac{2\gamma \eta_{k+1}}{\alpha}} \xi_{k+1}^{i,n}.
  \]
  Since $(\tilde{Z}_0^{i,n})$ are i.i.d with law $\mathcal{N}(0,\sigma_0^2 I_d)$, we deduce by induction that $\tilde{Z}_k^{i,n}$ are i.i.d with law $ \mathcal{N}(0,\sigma_k^2 I_d)$ and $\sigma_k$ satisfying $\sigma_0^2 \vee \frac{2\gamma}{\alpha} \geq \sigma_k^2 \geq \frac{\gamma}{\alpha}$
   For the second claim, we use the fact that the second moment of a $\chi^2(d)$ distribution is $d(d+2)$. The last claim follows from Markov's inequality and the independence of the $Z_k^{i,n}$.
\end{proof}

The next proposition bounds the discretization error.
We define 
\begin{equation}\label{eq:def_tk_kt}
t_k := \sum_{\ell=1}^k \eta_{\ell}, \qquad \text{and} \qquad  k_t := \sup\left\{ k \in \mathbb{N} \,:\, t_k \le t \right\}.
\end{equation}

\begin{proposition}
    \label{prop:euler}
    Let $T>0$ be a fixed time window. Let $\tilde{t} \geq 0$ and define $\tilde{k} := k_{\tilde{t}}$.
Then, there exists a solution $(X_t^{i,n})_{i \in [n]}$ of the SDE system \eqref{eq:SDEfinite} satisfying $X_0^{i,n} = X^{i,n}_{\tilde{k},\eta}$ for $i \in [n]$, such that for every $t\in [\tilde{t}, \tilde{t}+T]$ we have
\begin{equation*}
\begin{split}
      \esp{\norm{{X^{i,n}_{k_t,\eta} - X^{i,n}_{t-\tilde{t}}}}^2} \le \ex^{C^\mathrm{int}_{5,\alpha}T}\p{{C^\mathrm{int}_{4}}{\eta_{\tilde{k}}} + \frac{C^\mathrm{int}_{6,\alpha,T}}{n} +C^\mathrm{int}_{12,\alpha,T} n^{-\theta_d} }\,,
\end{split}
\end{equation*}
for $i \in [n]$, where $C^\mathrm{int}_{5,\alpha},C^\mathrm{int}_{6,\alpha,T},C^\mathrm{int}_{12,\alpha,T}$ depend on $f,\gamma$ and exponentially on $\alpha$ and where $C^\mathrm{int}_{4}$ depends on $f,\gamma$. 
If $\tilde k = 0$,  it holds in the same conditions:
\begin{equation*}
\begin{split}
    \esp{\norm{{X^{i,n}_{k_t,\eta} -X^{i,n}_{t}}}^2} \le \ex^{C^\mathrm{int}_{5,\alpha}T}\p{{C^\mathrm{int}_{4}}{\eta_{0}} + \frac{C^\mathrm{int}_{6,\alpha,T}}{n} }\,.
\end{split}
\end{equation*}
\end{proposition}
\begin{proof}
Consider the $n$ i.i.d standard Brownian motions $(B^{i,n}_t)_{i \in [n]}$ defined such that $B^{i,n}_{t_{k}-\tilde t} - B^{i,n}_{t_{k-1}-\tilde t} = \sqrt{\eta_k} \xi^{i,n}_k $ for every $k> \tilde{k}$. Define then $(X^{i,n}_{t})_{i \in [n]}$ to be the solutions of the SDE system \eqref{eq:SDEfinite} initialized at $\XO^{i,n} := X^{i,n}_{\tilde{k},\eta}$ and driven by $(B^{i,n}_t)_{i \in [n]}$.
  
Let us prove the first claim of the proposition. The second claim follows by the same argument, using the second claim of Proposition~\ref{prop:POC} instead of the first.
Define $e^i_t = X^{i,n}_{t-\tilde t}-X^{i,n}_{k_t,\eta}$.
Due to our specific choice of Brownian motions, we have for $k > \tilde{k}$,
\begin{equation*}
\begin{split}
  e^i_{{t_{k}}} &\!=\! X_{t_{k-1}}^{i,n} + \int_{t_{k-1}}^{t_{k}} \big( \clip_R(\theta(\mu_{s-\tilde{t}}^n)) - X_{s-\tilde{t}}^{i,n} \big)\dr s - (1-\eta_{k})X_{k-1,\eta}^{i,n} - \eta_{k} \clip_R(\theta(\mu_{k-1,\eta}^n) )\\
  &= e^i_{t_{k-1}}+ \int_{t_{k-1}}^{t_{k}}\p{\clip_R(\theta_\alpha(\mu_{s-\tilde t}^n)) - \clip_R(\theta_\alpha(\mu_{{k_s,\eta}}^n)) -e^i_s}\dr s.
\end{split}
\end{equation*}
Moreover, for $t \in [t_{k},t_{k+1})$, we have:
\begin{equation*}
\begin{split}
  e^i_t &= e^i_{t_{k}} + X^{i,n}_{t-\tilde{t}} - X^{i,n}_{t_{k}-\tilde{t}} = e^i_{t_{k}} + \int_{t_{k}}^t \big(\clip_R(\theta(\mu_{s-\tilde{t}}^n)) - X_{s-\tilde{t}}^{i,n} \big)\dr s + \int_{t_{k}}^t \sqrt{\frac{2\gamma}{\alpha}} dB_s\\
  &= e^i_{t_{k}} + \int_{t_{k}}^t \p{\clip_R(\theta_\alpha(\mu_{s-\tilde t}^n)) - \clip_R(\theta_\alpha(\mu_{{k_s,\eta}}^n)) -e^i_s}\dr s + \int_{t_{k}}^t \sqrt{\frac{2\gamma}{\alpha}} dB_s\\
  &\quad -(t-t_k)(X_{k,\eta}^{i,n} -  \clip_R(\theta_\alpha(\mu_{{k,\eta}}^n)) )
\end{split}
\end{equation*}
Combining the 2 previous combinations, we have for any $t \geq \tilde{t}$,
\begin{multline*}
    e^i_t = (X^{i,n}_{0} -X^{i,n}_{\tilde{k},\eta})+\int_{\tilde t}^t \p{\clip_R(\theta_\alpha(\mu_{s-\tilde t}^n)) - \clip_R(\theta_\alpha(\mu_{{k_s,\eta}}^n)) -e^i_s}\dr s  \\+ (t - t_{k_t})\p{X_{k_t,\eta}^{i,n} - \clip_R(\theta_\alpha(\mu_{{k_t,\eta}}^n))}  + \int_{t_{k_t}}^t \sqrt{\frac{2\gamma}{\alpha}} dB_s.
\end{multline*}
Therefore, noting that $e^i_{\tilde{t}} = 0$, we obtain
\begin{multline*}
    \bE \norm{e^i_t}^2 \le 5\int_{\tilde t}^t \bE\norm{e^i_s}^2\dr s + 5\int_{\tilde t}^t \bE\norm{\clip_R(\theta_\alpha(\mu_{s-\tilde t}^n)) - \clip_R(\theta_\alpha(\mu_{{k_s,\eta}}^n))}^2 \dr s\\+ 5\eta_{\tilde{k}}(2\eta_{\tilde{k}}(R^2 + \p{C_5}^2) + 2d\frac\gamma\alpha)\,.
\end{multline*}
We bound the second term as
\begin{equation*}
\begin{split}
         &\bE\norm{\clip_R(\theta_\alpha(\mu_{s-\tilde t}^n)) - \clip_R(\theta_\alpha(\mu_{{k_s,\eta}}^n))}^2\\
         &\le \bE\norm{\clip_R(\theta_\alpha(\mu_{s-\tilde t}^n)) - \clip_R(\theta_\alpha(\mu_{{k_s,\eta}}^n))}^2 \un_A + 4R^2\bP(A^c)\,,
\end{split}
\end{equation*}
 with the event $A$ defined as $A:= \{ \frac 1n\sum_{i=1}^n \norm{h_0(X_{s -\tilde t}^{i,n})} \le  2C_5 , \frac 1n \sum_{i=1}^n h_1(X_{s-\tilde t}^{i,n}) > \frac 12C^\mathrm{int}_{8,\alpha}, \frac 1n \sum_{i=1}^n h_1(X_{k_s,\eta}^{i,n})> \frac 12 C^\mathrm{int}_{8,\alpha} \}$. 
Using Lemma \ref{lem:lipC}, we obtain:
\[
\bE\norm{\clip_R(\theta_\alpha(\mu_{s-\tilde t}^n)) - \clip_R(\theta_\alpha(\mu_{{k_s-1,\eta}}^n))}^2 \un_A  \le \sum_{j=1}^n\frac{\bE\norm{e^j_s}^2}{n}\p{\frac{ L_{0,\alpha} + 4L_{1,\alpha}\frac{C_5}{C^\mathrm{int}_{8,\alpha}} } {\frac 12C^\mathrm{int}_{8,\alpha}}}^2,
\]
and note that, by symmetry of the particle system, $\frac{1}{n}\sum_{j=1}^n \bE\|e^j_s\|^2 = \bE \|e^i_s\|^2$.

It remains to bound $\bP(A^c)$. Let $(\bar X^{i,n}_t)_{i\in[n]}$ be, conditionally on $(X_0^{i,n})_{i\in[n]}$, the $n$ i.i.d. McKean--Vlasov solutions of~\eqref{eq:SDEfinite_bar} with initial law $\tilde{\nu}_0= \mu_0^n$. 
 Using Proposition~\ref{prop:POC}, we show that:
\begin{equation*}
\begin{split}
&\bP(\frac 1n \sum_{i=1}^n h_1(X^{i,n}_{k_s -1,\eta}) \le \frac 12 C^\mathrm{int}_{8,\alpha}) \le \frac{4}{(C^\mathrm{int}_{8,\alpha})^2} \bE\p{\frac 1n\sum_{i=1}^n h_1(X^{i,n}_{k_s -1,\eta}) -  \bE h_1(\bar X^{1,n}_{s-\tilde t})}^2 \\
&\le  \frac{12}{(C^\mathrm{int}_{8,\alpha})^2}\p{(L_{1,\alpha})^2\bE \norm{e^i_s}^2  + { (L_{1,\alpha})^2 \ex^{C^\mathrm{int}_{3,\alpha}(t-\tilde t) } (C^\mathrm{int}_{2,\alpha}n^{-1} + C^\mathrm{int}_{16}n^{-{\theta_d}})} + 2n^{-1}}\,.
\end{split}
\end{equation*}
Repeating the same calculation, we obtain:
\begin{equation*}
\begin{split}
       \bP(A^c) &\le  \frac{12}{(C^\mathrm{int}_{8,\alpha})^2}\p{(L_{1,\alpha})^2\bE \norm{e_s}^2  + { (L_{1,\alpha})^2 \ex^{C^\mathrm{int}_{3,\alpha}(t-\tilde t) } (C^\mathrm{int}_{2,\alpha}n^{-1} + C^\mathrm{int}_{16}n^{-{\theta_d}})} + 2n^{-1}}\\&+ \frac{8}{(C^\mathrm{int}_{8,\alpha})^2} \p{ (L_{1,\alpha})^2 \p{\ex^{C^\mathrm{int}_{3,\alpha}(t-\tilde t) } (C^\mathrm{int}_{2,\alpha}n^{-1} + C^\mathrm{int}_{16}n^{-{\theta_d}})} + 2n^{-1}}\\& + \frac{1}{(C_5)^2}\p{(L_{1,\alpha})^2 \ex^{C^\mathrm{int}_{3,\alpha}(t-\tilde t) } (C^\mathrm{int}_{2,\alpha}n^{-1} + C^\mathrm{int}_{16}n^{-{\theta_d}}) + 2(C_5)^2n^{-1}} + \frac {2d}n\,.
\end{split}
\end{equation*}
Consequently, we obtain that:
\[
\bE\norm{e^i_t}^2 \le  C^\mathrm{int}_{5,\alpha}\int_{\tilde t}^t \bE\norm{e^i_s}^2\dr s +C^\mathrm{int}_{4}\eta_{\tilde{k}}  + \frac{C^\mathrm{int}_{6,\alpha,T}}{n} +{C^\mathrm{int}_{12,\alpha,T}}{n^{-\theta_d}} \,,
\]
 with the constants $C^\mathrm{int}_4, C^\mathrm{int}_{5,\alpha},C^\mathrm{int}_{6,\alpha,T},C^\mathrm{int}_{12,\alpha,T}$ defined in Table~\ref{table:3}.
By the integral version of Grönwall’s lemma, the proof is complete.

\end{proof}

\subsubsection{Finalizing the Proof of Theorem~\ref{th:POC}}

Before proving the theorem, we state an intermediate lemma showing that the occupation measure of the particle system is close to a Gaussian measure.

\begin{lemma}
\label{lem:Gaussian}
Let $T > \frac 12 \log 2$ be a time window. Let $k\in \mathbb{N}$ such that $t_k \geq T$ and define $\tilde{k} = k_{t_k-T}$.
  Let $\cF$ be a $\sigma$--algebra on the set $\Omega$ such that 
    $(X^{i,n}_{\tilde{k},\eta})_{i \in [n]}$ is $\cF$--measurable. Then, there exists a $\cF$--measurable r.v. $g^n_k: \Omega \to \cP_2(\bR^d)$ such that:
    \[
    \bE (W_2(\mu_{k,\eta}^n,g^n_k )^2) \le  {4(C_5)^2\ex^{-2T}} +\frac{C^\mathrm{int}_{1,\alpha,T}}{n} + C^\mathrm{int}_{10,\alpha,T}\eta_{k} + C^\mathrm{int}_{13,\alpha,T}n^{-\theta_d}\,,
    \]
    with $ g^n_k = \mathcal N(m_k,\sigma_k^2I_d)$ and $m_k:\Omega \to\bR^d,\sigma_k\in\bR$, s. t. $\norm{m_k}\!\! \le\!\! C_5$ and $\sigma_0^2\ge\sigma_k^2\!\ge\! \frac \gamma{2\alpha}$ a. s..

    If $t_k \leq T$, the result still holds where this time $g^n_k$ is a deterministic  measure.
\end{lemma}
\begin{proof}

  Consider the occupation measure $\mu_{\tilde{k},\eta}^n$, which is $\F$-measurable. Following Proposition \ref{prop:POC} with $\nu_0 = \mu_{\tilde{k},\eta}^n$ (which satisfies the prerequisites thanks to Lemma \ref{lem:stabDIS}), consider, conditionnally on $\F$, the solutions $(\bar X_t^{i,n})_{i \in [n]}$ of the mean-field system~\eqref{eq:SDEfinite_bar} initialized i.i.d with law $\nu_0$. 
  By Proposition~\ref{prop:McKean} (with $y=0$), we have:
 \begin{equation}\label{eq:decomp_xbar}
         \bar X^{i,n}_T 
         = x_T + Z_T^{i,n} + \ex^{-T} \bar X^{i,n}_0  
 \end{equation}
  where $x_T$ is a function of $\nu_0$ and $Z_T^{1,n},\dots,Z_T^{n,n}$ are i.i.d variables of Gaussian law $\mathcal{N}(0,(1-\ex^{-2T})\frac{\gamma}{\alpha}I_d)$. Let $g_k^n$ be the measure $\mathcal{N}(x_T,(1-e^{-2T}) \frac{\gamma}{\alpha} I_d)$. Since $x_T$ is a $\F$-measurable random variable, so is $g_k^n$. Moreover, $\|x_T\| \leq R \leq C_5$ almost surely by Proposition \ref{prop:McKean}.
  Consider also, following Proposition \ref{prop:euler}, the solutions $(X^{i,n}_t)_{i \in [n]}$ of the continuous-time system \eqref{eq:SDEfinite} satisfying $X_0^{i,n} = X^{i,n}_{\tilde{k},\eta}$. We have
  \[
  W_2(\mu_{k,\eta}^n,g_k^n) \leq  W_2(\mu_{k,\eta}^n,\mu_{T}^n) + W_2(\mu_{T}^n, \bar \mu_{T}^n) +  W_2( \bar \mu_{T}^n, \mu_T^{n,Z}) + W_2( \mu_T^{n,Z},g_k^n), 
  \]
  where $\mu_T^{n,Z} := \frac 1n\sum_{i=1}^n \delta_{x_T+Z_T^{i,n}}$.
  We bound the first term in expectation using Proposition \ref{prop:euler} with $\tilde{t} = t_k - T$ (discretization error), and the second with Proposition \ref{prop:POC} (propagation of chaos). 
  Using \eqref{eq:decomp_xbar}, we bound $\bE W_2(\mu^{n,Z}_T ,\mu_{k,\eta}^n )^2 \leq e^{-2T} \bE \norm{\bar X^{1,n}_0 }^2 \le e^{-2T} (C_5)^2$. 
     Finally, since the variables $(Z_T^{i,n})_{i \in [n]}$ are i.i.d, we apply Lemma~\ref{lem:LLN} to get $\bE{W_2(\mu_T^{n,Z},g^n_k)^2}\le C^\mathrm{LLN}(2C_5^2 + 2\frac\gamma\alpha\sqrt{d(d+2)})n^{-\theta_d}$.


When $t_k \le T$, the same proof holds with the assumption that $X_{0,\eta}^{i,n}$ is a Gaussian vector (see Theorem~\ref{th:POC}'s assumptions) and using the second claim of Proposition~\ref{prop:POC}. 
\end{proof}

  We now finalize the proof. For initial distributions $\nu_0 =\mathcal N(m,\sigma^2 I_d)$ satisfying, conditionnally on a $\sigma$-algebra $\F$,
\begin{equation}
\label{eq:conditionu0}
      \norm{m} \le C_5, \text{ and }\, \frac \gamma{2\alpha}\le\sigma^2\le \sigma_0^2\,,
\end{equation}
by the assumption on $\gamma$, we obtain
$
\gamma >\frac{8\p{R \vee \norm{m-x^\ast}}}{\delta^{2 - 1} \kappa} \,.
$
Therefore the assumptions of Theorem~\ref{th:convergence} are verified for $\nu_0$ satisfying Equation~\eqref{eq:conditionu0}. It follows that the solutions $(\bar X^{i,n}_t)_{i\in[n]}$ of the mean-field equations~\eqref{eq:SDEfinite} initialized with $n$ i.i.d. random variables of law $\nu_0$ satisfy, conditionnally on $\F$:
\[
\bE_{\F}(\|\bar X_t^{i,n} -x^\ast\|^2)\le \frac {C_0}\alpha +   9{c_1}^{-1}\ex^{-c_1t}\int \norm{x-x^\ast}^2\dr\nu_0(x) \,.
\]
Defining $T^\mathrm{int}_{1} := \frac 1{c_1}\log  (72c_1^{-1}) $, and $\bar \mu_t^n := n^{-1}\sum_{i\in[n]}\delta_{\bar X_{t}^{i,n,\nu_0}}$, then for $t\ge T^\mathrm{int}_{1}$:
\begin{equation}
\label{eq:MF}
\bE W_2(\bar \mu_t^{n}, \delta_{x^\ast})^2\le \frac 18 \bE W_2(\nu_0,\delta_{x^\ast})^2 + \frac{C_0}{\alpha}\,.
\end{equation}

 We divide the time in windows of length $T^\mathrm{int}_{1}$. For $j\in \mathbb{N}$ and $s \in [0,T^\mathrm{int}_{1}]$, let us define $k^{(j)}_s : = k_{ s+jT^\mathrm{int}_{1}}$, and $\nu^{(j)}_s := \mu^n_{ k_s^{(j)},\eta}$ the empirical measure of the particles at iteration $k_s^{(j)}$. 
  We proceed to bound the distance $W_2(\nu_s^{(j)},\delta_{x^*})$ by induction on $j$.
  
  Let $\F_s^{(j)}$ be the $\sigma$-algebra generated by the variables $\left\{X^{i,n}_{k,\eta}\,:\, i\in [n],\,k\le k_s^{(j)}\right\}$. Let $T^\mathrm{int}_{2,\alpha} \ge \frac 12\log 2$ be a constant to be specified later; then by Lemma~\ref{lem:Gaussian}, there exists a measure $g_s^{(j)}$, which is Gaussian conditionally on $\F_s^{(j)}$, and a constant $P_{T^\mathrm{int}_{2,\alpha}}$, such that
  $
       \bE(W_2(\nu_s^{(j)}, g_s^{(j)})^2) \le P_{T^\mathrm{int}_{2,\alpha}}.
  $
   Following Proposition \ref{prop:POC}, we consider, conditionally on $\F_s^{(j)}$, the solutions $(\bar X^{i,n}_t)_{i \in [n]}$ of the mean-field equations \eqref{eq:SDEfinite_bar} with initial distribution~$\nu_0 = g_s^{(j)}$, and denote $\bar \mu_t^{n}$ the corresponding occupation measure. Consider also, applying Proposition \ref{prop:euler} with $\tilde{t} = s + j \Tint$, the solutions $(X^{i,n}_t)_{i \in [n]}$ to the continuous-time system \eqref{eq:SDEfinite} initialized as $X^{i,n}_0 = X^{i,n}_{ k_s^{(j)},\eta }$. Combining the bounds from these two propositions at time $t = \Tint$, we get
   \begin{equation}
   \label{eq:POC+EULER}
   \begin{split}
        \bE W_2( \nu_s^{(j+1)}, \bar \mu_{\Tint}^{n} )^2
        \le& \frac {C^\mathrm{int}_{11,\alpha}}{ n} + {C^\mathrm{int}_{15,\alpha}}\eta_{k_{jT^\mathrm{int}_{1}}}+ 4\ex^{C^\mathrm{int}_{3,\alpha}T^\mathrm{int}_{1}}\bE W_2(g_s^{(j)},\nu_s^{(j)})^2  + C^\mathrm{int}_{14,\alpha} n^{-\theta_d}\,,
   \end{split}
   \end{equation}
   where the constants $ C^\mathrm{int}_{11,\alpha},  C^\mathrm{int}_{14,\alpha}, C^\mathrm{int}_{15,\alpha}$ are defined in Table~\ref{table:3}. 
   
   Defining
  \(
    u_{s,j} := W_2(\nu_s^{(j)},\delta_{x^\ast})^2
  \), we have
  \begin{equation*}\begin{split}
    u_{s,j+1} &\le 2\bE W_2(\nu_s^{(j+1)}, \bar \mu^{n}_{T^\mathrm{int}_{1}})^2 + 2\bE W_2( \bar \mu^{n}_{T^\mathrm{int}_{1}} , \delta_{x^\ast})^2\,.
  \end{split}
  \end{equation*}
  We bound the first term using~\eqref{eq:POC+EULER}. Then, by Lemma \ref{lem:Gaussian}, we note that $g_s^{(j)}$ satisfies conditions~\eqref{eq:conditionu0} conditionally on $\F_s^{(j)}$. Therefore, we obtain by Eq.~\eqref{eq:MF} that 
  \(
  \bE W_2( \bar \mu^{n}_{T^\mathrm{int}_{1}} , \delta_{x^\ast})^2 \le \frac {C_0}{\alpha} + \frac 14 u_{s,j} + \frac 14\bE W_2(\nu_s^{(j)}, g_s^{(j)})\,.
  \)
  Consequently,
  \[
    u_{s,j+1} \!\!\le \frac { u_{s,j}}2 +2C^\mathrm{int}_{15,\alpha}\eta_{k_{j T^\mathrm{int}_{1}}} + \frac {2C^\mathrm{int}_{11,\alpha}}n +\frac{2C_{0}}{\alpha} + (\frac 12+8\ex^{C^\mathrm{int}_{3,\alpha}T^\mathrm{int}_{1}})P^\mathrm{int}_{2,\alpha} +2C^\mathrm{int}_{14,\alpha} n^{-\theta_d}.
  \]
  Then, choosing $T^\mathrm{int}_{2,\alpha} =\frac 12\log 2 \vee\frac 12\log\p{\alpha\frac {4(C_5)^2 (\frac 12+2\ex^{C^\mathrm{int}_{5,\alpha}T^\mathrm{int}_{1}})}{C_0}}$, we get $u_{s,j+1} \le \frac 12u_{s,j} + \frac{2C_0}{\alpha} + \frac {\eta_{jT^\mathrm{int}_{1}}}2C^\mathrm{int}_{7,\alpha} +\frac{C_{2,\alpha}}{2n} +\frac 12{C_{4,\alpha}} n^{-\theta_d} $ with the constants $C_{2,\alpha},C_{4,\alpha}, C^\mathrm{int}_{7,\alpha}$ defined in Table~\ref{table:1} and~\ref{table:3}.
    Iterating the latter equation, we obtain for every $j$ and $s\le T^\mathrm{int}_{1}$:
  \begin{equation*}
  u_{s,j} \le  \sum_{\ell=1}^{j-1} \frac{1}{2^{j-\ell}}\p{ C^\mathrm{int}_{7,\alpha}\eta_{k_{\ell T^\mathrm{int}_{1}}} } + \frac 1{2^j} (u_{s,0}+ {\eta_{0}}C^\mathrm{int}_{7,\alpha} ) + 6\frac{C_0}{\alpha} + C_{4,\alpha}n^{-\theta_d} +C_{2,\alpha}n^{-1}\,.
  \end{equation*}
  Then, with 
  $
  j^\mathrm{MF}_k:=\sup\{j\,:\,  \sum_{j=1}^{k}\eta_j \ge {jT^\mathrm{int}_{1}}  \}
  $, we obtain:
  \begin{equation}
  \label{eq:bound}
       \bE W_2(\mu_{k,\eta}^n,\delta_{x^\ast})^2 \le \sup_{s\le T^\mathrm{int}_{1}} {u_{s,j^\mathrm{MF}_k}}\,.
  \end{equation}

  By the definition of $j_k^\mathrm{MF}$, there exists some $t\le T^\mathrm{int}_{1}$ such that $k = k_{t+ j^\mathrm{MF}_kT^\mathrm{int}_{1}}$. Since the particles are exchangeable, we obtain that for every $i \in [n]$:
  \begin{equation}
  \label{eq:eqfinal}
          \esp{\norm{X_{k,\eta}^{i,n} -x^\ast}^2}
          \le \sum_{\ell=1}^{{j_k^\mathrm{MF}}-1} \frac{1}{2^{{j_k^\mathrm{MF}}-\ell}}\p{ C^\mathrm{int}_{7,\alpha}\eta_{k_{\ell T^\mathrm{int}_{1}}} } + \frac {C^\mathrm{int}_{18,\alpha}}{2^{j_k^\mathrm{MF}}}  + 6\frac{C_0}{\alpha}+\frac {C_{2,\alpha}}n+  C_{4,\alpha}n^{-\theta_d}\,,
  \end{equation}
  where $C^\mathrm{int}_{18,\alpha} := \sup_{s \le T^\mathrm{int}_{1}}u_{s,0}+ \eta_{0}C^\mathrm{int}_{7,\alpha}$.
   Defining the constant $C^\mathrm{int}_{9}$ with value in Table~\ref{table:3},
  we obtain:
  \(
   \sum_{\ell=1}^{j_k^\mathrm{MF}} 2^{-(j_k^\mathrm{MF} - \ell)} \eta_{k_{\ell T^\mathrm{int}_{1}}}\le  (j_k^\mathrm{MF})^{-\frac {\zeta}{1-\zeta}}\eta_0C^\mathrm{int}_{9}\sum_{\ell=0}^\infty 2^{-\ell}(\ell+1)^{\frac {\zeta}{1-\zeta}}\,.
  \)
  Using Lemma~\ref{lem:step-size-bound}:
  \(
  j_k^\mathrm{MF} \ge \eta_0\frac{ k^{1-\zeta}-1}{(1-\zeta)T^\mathrm{int}_{1}} -1
  \).
  Then:
  \(
  (j_k^\mathrm{MF})^{-\frac {\zeta}{1-\zeta}} \le k^{-\zeta}\p{\frac{\eta_0}{2(1-\zeta)T^\mathrm{int}_{1}}}^{-\frac \zeta{1-\zeta}}
  \),
  for $k\ge K_0$ where:
  \(
  K_0 := (2+\frac {2\eta_0}{(1-\zeta)T^\mathrm{int}_{1}})^{\frac 1{1-\zeta} } 
  \).
  As a consequence, we obtain: 
  \begin{equation}
  \label{eq:sum}
  \sum_{\ell=1}^{j_k^\mathrm{MF}-1} {2^{-(j_k^\mathrm{MF}-\ell)}}{\eta_{\ell T^\mathrm{int}_{1}}} \le \frac{1}{k^\zeta}\p{ \eta_0C^\mathrm{int}_{9} \p{\frac{\eta_0}{2(1-\zeta)T^\mathrm{int}_{1}}}^{-\frac \zeta{1-\zeta}}}\sum_{\ell=0}^\infty 2^{-\ell}(\ell+1)^{\frac {\zeta}{1-\zeta}}\,.
  \end{equation}
  Using Equation~\eqref{eq:sum} in Equation~\eqref{eq:eqfinal}, the proof is concluded.

  The second claim of the theorem is obtained by combining the second claims of Propositions~\ref{prop:POC} and~\ref{prop:euler} together with Theorem~\ref{th:convergence}. 
  
\subsection{Proof of Theorem~\ref{th:best}}
\label{sec:propTH2}
Under Assumption~\ref{hyp:f-flot}, it holds for some constants $L,\tilde \delta$: $f(x) -f(x_*) \le L \norm{x-x_*}^2$ on $x\in B(x^\ast,\tilde \delta)$. Let $C>0$ be a generic constant independent of $\alpha$, we obtain that
\begin{equation}
\label{eq:computation}
    \begin{split}
&\bE{\sup_{y\in \mathcal X_k^\ast}\norm{y -x^\ast }} \le\sqrt{\frac 2\kappa}\bE{\sup_{y\in \mathcal X_k^\ast}\p{f(X_k^\ast) -f(x^\ast)}^{\frac 12}}  = \sqrt{\frac 2\kappa}\bE{\min_{i\in[n]}({f(X^{i,n}_{k,\eta}) -f(x^\ast)})^{\frac 12}}\\
&\le \sqrt{\frac 2\kappa}\p{\esp{\min_{i\in[n]}L{\norm{X^{i,n}_{k,\eta} -x^\ast}}}+ C\bP(\min_{i\in[n]}{\norm{X^{i,n}_{k,\eta} -x^\ast}\ge \tilde \delta})}\\
&\le C \esp{{\min_{i\in[n]}\norm{X^{i,n}_{k,\eta} -x^\ast}}}\,.
\end{split}
\end{equation}

Let \( T_\alpha  \) be a constant depending on \( \alpha \) such that, for any Gaussian measure 
\( g = \mathcal N(m, \sigma^2 I_d) \) with \( \|m\| \le C_5 \) and 
\( \frac{\gamma}{2\alpha} \le \sigma^2 \le \sigma_0^2 \), 
the solution \( \bar X_t \) of the McKean--Vlasov equation~\eqref{eq:SDE-clip} 
initialized with \( g \) satisfies, by Proposition \ref{prop:McKean} and Theorem~\ref{th:convergence}, that 
\( \bar X_{T_\alpha} \) is a Gaussian random vector of law $\mathcal{N}(m_{T_\alpha}, \sigma_{T_\alpha})$ such that (since \( T_\alpha \) is taken large enough):
\begin{equation}
\label{eq:gaussCriterium}
    \left| m_{T_\alpha} - x^\ast\right|^2 \le \frac{2C_0}{\alpha}
    \quad \text{and} \quad
    \frac{\gamma}{2\alpha} \le \sigma_{T_\alpha}^2 \le \frac{2\gamma}{\alpha}\,.
\end{equation}
Assume that \( k \ge k_{T_\alpha} \), and let \( \delta = k - k_{t_k - T_\alpha} \). 
Following Lemma~\ref{lem:Gaussian}, let \( g \) be a random measure which, conditionally on the $\sigma$-algebra 
\( \mathcal F := \sigma(X_{j,\eta}^{i,n} : i \le n,\, j \le k - \delta) \), 
is Gaussian \( \mathcal N(m, \sigma^2 I_d) \) with 
\( \|m\| \le C_5 \) and \( \frac{\gamma}{2\alpha} \le \sigma^2 \le \sigma_0^2 \) and which satisfies
\(
    \mathbb{E}[ W_2(\mu_{k-\delta,\eta}^n, g)^2] 
    \le P_1
\)
for some constant $P_1$.
By Propositions~\ref{prop:POC} and~\ref{prop:euler} together with Lemma~\ref{lem:LLN}, we show that, conditionally on \( \mathcal F \), 
the law $\bar \rho_t$ of the solution of the McKean--Vlasov SDE~\eqref{eq:SDE-clip} 
initialized with \( g \) satisfies for some small bound $P_2$ given by Proposition~\ref{prop:POC} and~\ref{prop:euler}:
\(
    \bE W_2(\mu_{k,\eta}^n,\bar \rho_{T_\alpha})^2 
    \le C_\alpha(n^{-\theta_d} + P_2+  P_1) =: P_3
\), for some constant $C_\alpha$ depending on $\alpha$.


We can choose \( n \) large enough with respect to \( \alpha \), and \( k \) large enough with respect to \( n \) and \( \alpha \), such that 
\(
    P_3 \le C_\alpha\, n^{-\frac {\theta_d}{2(C_{3,\alpha}^\mathrm{int}+ C_{5}^\mathrm{int}+ 1)}},
\)
for some constant \( C_\alpha \) that does not depend on \( n \) or \( k \).
Moreover, conditionally on \( \mathcal F \), the random measure \( (\bar \rho_{T_\alpha}) \) is a Gaussian measure 
\( \mathcal N(m, \sigma^2 I_d) \) satisfying~\eqref{eq:gaussCriterium}.
Consequently, for some constants $C,c$ independent of $\alpha$:
$
    \bar \rho_{T_\alpha}(B(x^\ast,\frac r2)) \le \frac 4C(r\sqrt\alpha)^d\ex^{-c\alpha r^2}
$\,.
 Continuing the computation~\eqref{eq:computation}, we obtain
\begin{equation*}
    \esp{{\min_{i\in[n]}\norm{X^{i,n}_{k,\eta} -x^\ast}} } \le r +  \esp{\norm{ X_{k,\eta}^{i,n} -x^\ast}^2}^{\frac 12} \bP(\min_{i\in [n]}\norm{ X_{k,\eta}^{i,n} -x^\ast} >r)^{\frac 12} \,.
\end{equation*}
Let $f_r$ be a $\frac 2 r$-Lipschitz function such that $ \un_{B(x^\ast,\frac r2)}\le f_r \le \un_{B(x^\ast,r)} $. By Kantorovich-Rubenstein duality~\cite[Remark 6.4]{villani2008optimal} applied to $f_r$ it holds:
\begin{equation*}
\begin{split}
        \bP(\min_{i\in [n]}\norm{ X_{k,\eta}^{i,n} -x^\ast} >r)&= \bP (\mu_{k,\eta}^n(B(x^\ast,r )) =0) \\&\le \bP( \bar \rho_{T_\alpha}(B(x^\ast,\frac r2) -\frac 2rW_1(\mu_{k,\eta}^n,\bar \rho_{T_\alpha}) \le 0) \\
        & \le \esp{\bE(\frac{4W_2(\mu_{k,\eta}^n,\bar \rho_{T_\alpha})^2}{r^2\bar \rho_{T_\alpha}(B(x^\ast,\frac r2))^2}|\cF)}\\
        & \le C \ex^{-c\alpha r^2}\frac {P_3}{r^{2d+2}\alpha^d}\,.
\end{split}    
\end{equation*}
Taking $r =\alpha^{-\frac d{2(d+2)}}{n^{-\frac {\theta_d}{4(C_{3,\alpha}^\mathrm{int}+ C_{5}^\mathrm{int}+ 1)(d+2)}}}$, we finish the proof with Equation~\eqref{eq:computation}.


\appendix
\section{Technical lemmas} 
\label{sec:add}
We state here some classical technical results, whose proofs are omitted due to space limitations.

\begin{lemma}
    \label{lem:step-size-bound}
    Let $k\ge 1$, $\eta_0 \in(0,1]$, $\zeta\in [0,1)$ and $\eta_k :=\frac{\eta_0}{k^\zeta} $, then we obtain 
    \(
    \frac { (k+1)^{1-\zeta} -1  }{1-\zeta} \le  \frac{t_k}{\eta_0}\le \frac {k^{1-\zeta}}{1-\zeta} 
    \).
    And, we  have \(
     ((1-\zeta)\frac{t}{\eta_0} +1)^{\frac 1{1-\zeta}} -1 \le k_t \le  (\frac{t}{\eta_0}(1-\zeta))^{\frac 1{1-\zeta}}+1
    \).
\end{lemma}

The following lemma is extracted from \cite[Theorem~1]{fournier2015rate}.
\begin{lemma}
\label{lem:LLN}
Let $\mu \in \mathcal{P}_4(\mathbb{R}^d)$, and let $(X^i)_{i\in[n]}$ be i.i.d. random variables with law $\mu$. Then, for every $n\in\bN^*$:
\(
\mathbb{E}\!\left[ W_2\!\left(\mu, \frac{1}{n}\sum_{i\in [n]} \delta_{X^i}\right)^2 \right]
  \le C^\mathrm{LLN}(\int \norm{x}^4\dr\mu(x))^{1/2}\, n^{-\theta_d}
\),
where $\theta_d := \tfrac{1}{3} \land \tfrac{2}{d}$ and $C^\mathrm{LLN}$ is a constant depending on $d$ not explicited here.
\end{lemma}

\section{Definition of the constants}
\label{sec:toc}
In this section, we define the constants appearing in our main results and proofs. Table~\ref{table:1} lists the constants used in our main results.  
Table~\ref{table:2} lists the constants related to time \(t\) or the number of iterations \(k\).  
Table~\ref{table:3} lists the constants used in intermediate results. 

\begin{table}[ht]
    \label{table:1}
\centering
\begin{tabular}{|c|c|}
\hline
\textbf{Constant} & \textbf{Expression} \\
\hline
$C_{0}$ & $ 6\gamma d$\\
\hline
$C_{1,\alpha}$ & $2^{\frac{\eta_0}{(1-\zeta)T^\mathrm{int}_{1}}+1}\p{3((C_5)^2 +\frac {C_0}{\alpha_0} + \frac{\gamma d}{\alpha_0}) + C^\mathrm{int}_{7,\alpha}\eta_0}$ \\
\hline
$C_{2,\alpha}$ & $4C^\mathrm{int}_{11,\alpha} +(1+16\ex^{C^\mathrm{int}_{5,\alpha}T^\mathrm{int}_{1}})C^\mathrm{int}_{1,\alpha,T^\mathrm{int}_{2,\alpha}} $\\
\hline
$C_{3,\alpha}$ & $  C^\mathrm{int}_{7,\alpha} \eta_0C^\mathrm{int}_{9}\p{\frac{\eta_0}{2(1-\zeta)T^\mathrm{int}_{1}}}^{-\frac \zeta{1-\zeta}} \sum_{\ell=0}^\infty 2^{-\ell}(\ell+1)^{\frac {\zeta}{1-\zeta}}$ \\
\hline
$C_{4,\alpha}$& $(1+16\ex^{C^\mathrm{int}_{5,\alpha}T^\mathrm{int}_{1}})C^\mathrm{int}_{13,\alpha}+4C^\mathrm{int}_{14,\alpha}$\\
\hline
$C_5$ & $\sqrt{(2\bE_{ X_0\sim\nu}\norm{X_0}^2 + 2R^2 + (d+1)\sigma_0^2)\vee \p{R^2 + 2d\frac\gamma\alpha} )}$ 
\\ 

\hline
$c_1$ & $  \frac{1}{ (1+\frac 2{\lambda \gamma})}$\\

\hline
$c_2$ & $2^{\frac{\eta_0}{(1-\zeta)T^\mathrm{int}_{1}}}-1$ \\
\hline
$
c_{3,\alpha}
$
&
$\frac {\theta_d}{4(C_{3,\alpha}^\mathrm{int}+ C_{5}^\mathrm{int}+ 1)(d+2)}$
\\ 
\hline
$\alpha_0 $ & $ \frac{2(C^\mathrm{Lap})^2}{c_1^2\gamma d}$ \\
\hline
$\bar \gamma$ & $\frac{4(R\vee \norm{\bE({X_0})-x^\ast})}{\kappa\delta}$ \\
\hline
$\tilde \gamma_0$ &  $ \frac{8(R+C_5)}{\kappa\delta}$\\
\hline
$\theta_d$ & $\frac 13 \vee \frac 2d $\\
\hline

\end{tabular}
\caption{Constants used in the main results}
\end{table}

\begin{table}[ht]
    \label{table:2}
\centering
\begin{tabular}{|c|c|}
\hline
\textbf{Constant} & \textbf{Expression} \\
\hline
$K_0 $&$(2+\frac {2\eta_0}{(1-\zeta)T^\mathrm{int}_{1}})^{ 1/{1-\zeta} } $\\
\hline
$T^\mathrm{int}_{1} $ &$ \frac 1{c_1}\log  (72c_1^{-1}) $\\
\hline
$T^\mathrm{int}_{2,\alpha}$ & $\frac 12\log 2 \vee\frac 12\log\p{\alpha\frac {4(C_5)^2 (\frac 12+2\ex^{C^\mathrm{int}_{5,\alpha}T^\mathrm{int}_{1}})}{C_0}}$\\
\hline
\end{tabular}
\caption{Constants related to time $t$ or the number of iterations $k$
}
\end{table}

\begin{table}[ht]
    \label{table:3}
\centering
\begin{tabular}{|c|c|}
\hline
\textbf{Constant} & \textbf{Expression} \\
\hline
$C^\mathrm{int}_{1,\alpha,T}$ &$2C^\mathrm{int}_{6,\alpha,T} \ex^{C^\mathrm{int}_{5,\alpha}T}+2\ex^{C^\mathrm{int}_{3,\alpha}T}C^\mathrm{int}_{2,\alpha}$\\
\hline
$C^\mathrm{int}_{2,\alpha}$ & $ \frac {32}{\p{C^\mathrm{int}_{8,\alpha}}^2}\p{2R^2 +(C_5)^2 +\p{\frac{C_5}{C^\mathrm{int}_{8,\alpha}}}^2  } +  {8R^2}{ \p{1+\frac{12}{(C^\mathrm{int}_{8,\alpha})^2}} } + 8R^2 d$ \\
\hline
$C^\mathrm{int}_{3,\alpha}$ & $2\p{2\p{\frac{L_{0,{\alpha}} + 4\frac {C_5}{C^\mathrm{int}_{8,\alpha}} L_{1,{\alpha}}  }{C^\mathrm{int}_{8,\alpha} }}^2+ \frac{16R^2L_{1,{\alpha}}^2}{(C^\mathrm{int}_{8,\alpha})^2}}$ \\
\hline
$C^\mathrm{int}_{4}$ & $ 5(2\eta_0(R^2 + \p{C_5}^2) + 2\frac{d \gamma}{\alpha_0})$ \\
\hline
$C^\mathrm{int}_{5,\alpha}$ & $5\p{\frac{ L_{0,\alpha} + 4L_{1,\alpha}\frac{C_5}{C^\mathrm{int}_{8,\alpha}} } {\frac 12C^\mathrm{int}_{8,\alpha}}}^2 +  \frac{240L_{1,\alpha}^2R^2}{(C^\mathrm{int}_{8,\alpha})^2}$ \\
\hline
$C^\mathrm{int}_{6,\alpha,T}$ & $R^2C^\mathrm{int}_{2,\alpha}\ex^{C^\mathrm{int}_{3,\alpha}T}\p{400 (\frac{L_{1,\alpha}}{C^\mathrm{int}_{8,\alpha}})^2 + 20(\frac{L_{1,\alpha}}{C_5})^2  } +\frac{800R^2}{(C^\mathrm{int}_{8,\alpha})^2}+40(1+d)R^2$ \\
\hline 
$C^\mathrm{int}_{7,\alpha}$&$ 4C^\mathrm{int}_{15,\alpha} +( 1+16\ex^{C^\mathrm{int}_{5,\alpha}T^\mathrm{int}_{1}})C^\mathrm{int}_{10,\alpha,T^\mathrm{int}_{2,\alpha}}$\\
\hline
$C^\mathrm{int}_{8,\alpha}$ & $\frac{3}{4}\exp\p{-\alpha\sup_{y\in B(0,2C_5) } f(y) }$ \\
\hline
$C^\mathrm{int}_{9}$ &\!\!\! $\min\p{(T^\mathrm{int}_{1}(1-\zeta))^{\frac {\zeta}{1-\zeta}},T^\mathrm{int}_{1}(1-\zeta), (1+T^\mathrm{int}_{1}(1-\zeta))^{\frac \zeta{1-\zeta}}-1}^{-1}$\!\!\! \\
\hline
$C^\mathrm{int}_{10,\alpha,T} $&$2C^\mathrm{int}_4\ex^{C^\mathrm{int}_{5,\alpha}T}\p{2+ 2^{\frac \zeta {1-\zeta}}((\frac {(1-\zeta)T}{\eta_0})^{\frac 1{1-\zeta}}+1)}^\zeta$ \\
\hline
$C^\mathrm{int}_{11,\alpha}$&$\ex^{C^\mathrm{int}_{3,\alpha}T^\mathrm{int}_{1}}C^\mathrm{int}_{2,\alpha}+ \ex^{C^\mathrm{int}_{5,\alpha}T^\mathrm{int}_{1}}C^\mathrm{int}_{6,\alpha,T^\mathrm{int}_{1}} $\\
\hline
$C^\mathrm{int}_{12,\alpha,T}$&$R^2 C^\mathrm{int}_{16}\ex^{C^\mathrm{int}_{3,\alpha}T}\p{400 (\frac{L_{1,\alpha}}{C^\mathrm{int}_{8,\alpha}})^2 + 20(\frac{L_{1,\alpha}}{C_5})^2  } $\\
\hline
$C^\mathrm{int}_{13,\alpha,T}$& $2C^\mathrm{int}_{12,\alpha,T} \ex^{C^\mathrm{int}_{5,\alpha}T}+2\ex^{C^\mathrm{int}_{3,\alpha}T}C^\mathrm{int}_{16} + 2 C^{\mathrm{LLN}}(2C_5^2 + 2\frac\gamma\alpha\sqrt{d(d+2)})$\\
\hline
$C^\mathrm{int}_{14,\alpha}$& $2C^\mathrm{int}_{12,\alpha,T^\mathrm{int}_\alpha} \ex^{C^\mathrm{int}_{5,\alpha}T^\mathrm{int}_\alpha} +2C^\mathrm{int}_{16}$\\
\hline
$C^\mathrm{int}_{15,\alpha}$&$2\ex^{C^\mathrm{int}_{5,\alpha}T^\mathrm{int}_{1}}C^\mathrm{int}_{4}$\\
\hline
$C^\mathrm{int}_{16}$&$3C^\mathrm{LLN}C^\mathrm{int}_{17}$\\
\hline
$C^\mathrm{int}_{17}$&$2(C_5)^2 +2\sigma_0^2\sqrt{d(d+2)}$\\
\hline
$C^\mathrm{Lap}$ & Laplace principle's constant depending on $f$ defined in Lemma~\ref{lem:laplace}\\
\hline
$\delta,\lambda,\kappa,a$ & Positive constants defined in Assumption~\ref{hyp:f-flot}\\
\hline 
$L_{0,\alpha}$ &$\alpha\ex^{-\alpha f(x^\ast)} \sup_{x\in\bR^d} \ex^{-\frac {\kappa\alpha}2 \norm{x-x^\ast}^2}(1+L\norm{x} +L\norm{x}^{a+1})$  \\
\hline
$L_{1,\alpha}$ &$\alpha\ex^{-\alpha f(x^\ast)} \sup_{x\in\bR^d} \ex^{-\frac {\kappa\alpha}2 \norm{x-x^\ast}^2}L(1 +\norm{x}^{a})$\\
\hline
$C^\mathrm{LLN}$ & Constant of \cite[Theorem 1]{fournier2015rate}\\
\hline
\end{tabular}
\caption{Constants used in intermediate results}
\end{table}

\bibliography{bib}

\newcommand{\etalchar}[1]{$^{#1}$}
\begin{thebibliography}{GHKV25}

\bibitem[ABMS25]{almi2025general}
S.~Almi, A.~Baldi, M.~Morandotti, and F.~Solombrino.
\newblock A general perspective on cbo methods with stochastic rate of
  information.
\newblock {\em arXiv preprint arXiv:2507.20029}, 2025.

\bibitem[ASJ13]{AskariSichani2013}
O.~Askari-Sichani and M.~Jalili.
\newblock Large-scale global optimization through consensus of opinions over
  complex networks.
\newblock {\em Complex Adaptive Systems Modeling}, 1(1):11, 2013.

\bibitem[BC11]{bauschke2020correction}
Heinz~H Bauschke and Patrick~L Combettes.
\newblock Convex analysis and monotone operator theory in hilbert spaces.
\newblock {\em CMS Books in Mathematics, Ouvrages de math{\'e}matiques de la
  SMC (}, 2011.

\bibitem[BEZ25]{bayraktar2025uniform}
E.~Bayraktar, I.~Ekren, and H.~Zhou.
\newblock Uniform-in-time weak propagation of chaos for consensus-based
  optimization.
\newblock {\em arXiv preprint arXiv:2502.00582}, 2025.

\bibitem[BGP23]{borghi2023consensus}
G.~Borghi, S.~Grassi, and L.~Pareschi.
\newblock Consensus based optimization with memory effects: random selection
  and applications.
\newblock {\em Chaos, Solitons \& Fractals}, 174:113859, 2023.

\bibitem[BHK{\etalchar{+}}22]{bae2022constrained}
H.-O. Bae, S.-Y. Ha, M.~Kang, H.~Lim, C.~Min, and J.~Yoo.
\newblock A constrained consensus based optimization algorithm and its
  application to finance.
\newblock {\em Applied Mathematics and Computation}, 416:126726, 2022.

\bibitem[BHP22]{borghi2022consensus}
G.~Borghi, M.~Herty, and L.~Pareschi.
\newblock A consensus-based algorithm for multi-objective optimization and its
  mean-field description.
\newblock In {\em 2022 IEEE 61st Conference on Decision and Control (CDC)},
  pages 4131--4136. IEEE, 2022.

\bibitem[BHP25]{bianchi2024longrunconvergencediscretetime}
P.~Bianchi, W.~Hachem, and V.~Priser.
\newblock Long run convergence of discrete-time interacting particle systems of
  the mckean-vlasov type.
\newblock {\em To be published in Stochastic Processes and Applications}, 2025.

\bibitem[CCTT18]{carrillo2018analytical}
J.~A Carrillo, Y.-P. Choi, C.~Totzeck, and O.~Tse.
\newblock An analytical framework for consensus-based global optimization
  method.
\newblock {\em Mathematical Models and Methods in Applied Sciences},
  28(06):1037--1066, 2018.

\bibitem[CD22]{Chaintron_2022}
L.-P. Chaintron and A.~Diez.
\newblock Propagation of chaos: A review of models, methods and applications.
  i. models and methods.
\newblock {\em Kinetic and Related Models}, 15(6):895, 2022.

\bibitem[CHSV21]{carrillo2021consensus}
J.~A. Carrillo, F.~Hoffmann, A.~M Stuart, and U.~Vaes.
\newblock Consensus based sampling.
\newblock {\em arXiv preprint arXiv:2106.02519}, 2021.

\bibitem[FC95]{fogel2006evolutionary}
D.~B Fogel and Evolutionary Computation.
\newblock Toward a new philosophy of machine intelligence.
\newblock {\em IEEE Evolutionary Computation}, 1080, 1995.

\bibitem[FG15]{fournier2015rate}
N.~Fournier and A.~Guillin.
\newblock On the rate of convergence in wasserstein distance of the empirical
  measure.
\newblock {\em Probability theory and related fields}, 162(3):707--738, 2015.

\bibitem[FHPS20]{fornasier2020consensus}
M.~Fornasier, H.~Huang, L.~Pareschi, and P.~S{\"u}nnen.
\newblock Consensus-based optimization on hypersurfaces: Well-posedness and
  mean-field limit.
\newblock {\em Mathematical Models and Methods in Applied Sciences},
  30(14):2725--2751, 2020.

\bibitem[FKR24]{Fornasier_2024}
M.~Fornasier, T.~Klock, and K.~Riedl.
\newblock Consensus-based optimization methods converge globally.
\newblock {\em SIAM Journal on Optimization}, 34(3):2973–3004, September
  2024.

\bibitem[FRRS25]{fornasier2025consensus}
M.~Fornasier, P.~Richt{\'a}rik, K.~Riedl, and L.~Sun.
\newblock Consensus-based optimisation with truncated noise.
\newblock {\em European Journal of Applied Mathematics}, 36(2):292--315, 2025.

\bibitem[FS25]{fornasier2025pde}
M.~Fornasier and L.~Sun.
\newblock A pde framework of consensus-based optimization for objectives with
  multiple global minimizers.
\newblock {\em Communications in Partial Differential Equations},
  50(4):493--541, 2025.

\bibitem[GHKV25]{gerber2025uniform}
N.~Gerber, F.~Hoffmann, D.~Kim, and U.~Vaes.
\newblock Uniform-in-time propagation of chaos for consensus-based
  optimization.
\newblock {\em arXiv preprint arXiv:2505.08669}, 2025.

\bibitem[GHV23]{gerber2023mean}
N.~J. Gerber, F.~Hoffmann, and U.~Vaes.
\newblock Mean-field limits for consensus-based optimization and sampling.
\newblock {\em arXiv preprint arXiv:2312.07373}, 2023.

\bibitem[GP21]{grassi2021particle}
S.~Grassi and L.~Pareschi.
\newblock From particle swarm optimization to consensus based optimization:
  stochastic modeling and mean-field limit.
\newblock {\em Mathematical Models and Methods in Applied Sciences},
  31(08):1625--1657, 2021.

\bibitem[Has70]{hastings1970monte}
W.~K. Hastings.
\newblock Monte {C}arlo sampling methods using {M}arkov chains and their
  applications.
\newblock {\em Biometrika}, 57(1):97--109, 1970.

\bibitem[HJK20]{ha2020convergence}
S.-Y. Ha, S.~Jin, and D.~Kim.
\newblock Convergence of a first-order consensus-based global optimization
  algorithm.
\newblock {\em Mathematical Models and Methods in Applied Sciences},
  30(12):2417--2444, 2020.

\bibitem[HK25]{huang2025uniform}
H.~Huang and H.~Kouhkouh.
\newblock Uniform-in-time mean-field limit estimate for the consensus-based
  optimization.
\newblock {\em ESAIM: Control, Optimisation and Calculus of Variations}, 31:69,
  2025.

\bibitem[Hol75]{holland1992adaptation}
J.~H. Holland.
\newblock {\em Adaptation in natural and artificial systems. An introductory
  analysis with applications to biology, control, and artificial intelligence}.
\newblock University of Michigan Press, Ann Arbor, Mich., 1975.

\bibitem[HQ22]{huang2022mean}
H.~Huang and J.~Qiu.
\newblock On the mean-field limit for the consensus-based optimization.
\newblock {\em Mathematical Methods in the Applied Sciences},
  45(12):7814--7831, 2022.

\bibitem[HQR24]{huang2024consensus}
H.~Huang, J.~Qiu, and K.~Riedl.
\newblock Consensus-based optimization for saddle point problems.
\newblock {\em SIAM Journal on Control and Optimization}, 62(2):1093--1121,
  2024.

\bibitem[KE95]{Kennedy05Particle}
J.~Kennedy and R.~Eberhart.
\newblock Particle swarm optimization.
\newblock In {\em Proceedings of ICNN'95 - International Conference on Neural
  Networks}, volume~4, pages 1942--1948 vol.4, 1995.

\bibitem[KHJK22]{ko2022convergence}
D.~Ko, S.-Y. Ha, S.~Jin, and D.~Kim.
\newblock Convergence analysis of the discrete consensus-based optimization
  algorithm with random batch interactions and heterogeneous noises.
\newblock {\em Mathematical Models and Methods in Applied Sciences},
  32(06):1071--1107, 2022.

\bibitem[Kir10]{kirwin2010higherasymptoticslaplacesapproximation}
W.~D Kirwin.
\newblock Higher asymptotics of laplace's approximation.
\newblock {\em Asymptotic Analysis}, 70(3-4):231--248, 2010.

\bibitem[KST24]{klamroth2024consensus}
K.~Klamroth, M.~Stiglmayr, and C.~Totzeck.
\newblock Consensus-based optimization for multi-objective problems: a
  multi-swarm approach.
\newblock {\em Journal of Global Optimization}, 89(3):745--776, 2024.

\bibitem[MRC87]{meleard1987propagation}
S.~M{\'e}l{\'e}ard and S.~Roelly-Coppoletta.
\newblock A propagation of chaos result for a system of particles with moderate
  interaction.
\newblock {\em Stochastic processes and their applications}, 26:317--332, 1987.

\bibitem[PC{\etalchar{+}}19]{peyre2019computational}
G.~Peyr{\'e}, M.~Cuturi, et~al.
\newblock Computational optimal transport: With applications to data science.
\newblock {\em Foundations and Trends{\textregistered} in Machine Learning},
  11(5-6):355--607, 2019.

\bibitem[Pel98]{pelletier1998weak}
M.~Pelletier.
\newblock Weak convergence rates for stochastic approximation with application
  to multiple targets and simulated annealing.
\newblock {\em Annals of Applied Probability}, pages 10--44, 1998.

\bibitem[PKB07]{Poli2007}
R.~Poli, J.~Kennedy, and T.~Blackwell.
\newblock Particle swarm optimization.
\newblock {\em Swarm Intelligence}, 1(1):33--57, 2007.

\bibitem[PTTM17]{pinnau2017consensus}
R.~Pinnau, C.~Totzeck, O.~Tse, and S.~Martin.
\newblock A consensus-based model for global optimization and its mean-field
  limit.
\newblock {\em Mathematical Models and Methods in Applied Sciences},
  27(01):183--204, 2017.

\bibitem[PV02]{Parsopoulos2002}
K.~E. Parsopoulos and M.~N. Vrahatis.
\newblock Recent approaches to global optimization problems through particle
  swarm optimization.
\newblock {\em Natural Computing}, 1(2-3):235--306, 2002.

\bibitem[Szn84]{sznitman1984nonlinear}
A.-S. Sznitman.
\newblock Nonlinear reflecting diffusion process, and the propagation of chaos
  and fluctuations associated.
\newblock {\em Journal of functional analysis}, 56(3):311--336, 1984.

\bibitem[TW20]{totzeck2020consensus}
C.~Totzeck and M.-T. Wolfram.
\newblock Consensus-based global optimization with personal best.
\newblock {\em Mathematical Biosciences and Engineering}, 17(5):6026--6044,
  2020.

\bibitem[Vil09]{villani2008optimal}
C.~Villani.
\newblock {\em Optimal transport}, volume 338 of {\em Grundlehren der
  mathematischen Wissenschaften [Fundamental Principles of Mathematical
  Sciences]}.
\newblock Springer-Verlag, Berlin, 2009.
\newblock Old and new.

\end{thebibliography}
\bibliographystyle{alpha}

\end{document}